\documentclass[11pt, reqno]{amsart}
\usepackage{amsmath,amssymb,amsfonts,amsthm,graphicx,mathrsfs,mathtools}
\usepackage{array}
\usepackage[margin=1.5in]{geometry}
\usepackage[dvipsnames]{xcolor}
\usepackage{graphicx,tipa}
\usepackage{tikz}
\usepackage{graphicx,amsmath,calc}
\usepackage{subcaption}
\makeatletter
\@namedef{subjclassname@2020}{%
\textup{2020} Mathematics Subject Classification}
\makeatother

\title[Complete minimal surfaces in $\mathbb{R}^4$]{Minimal surfaces in $\mathbb{R}^4$ with three planar ends\\
I: Holomorphicity and uniqueness}
\author[J. Lee]{Jaehoon Lee}
\address[]{Jaehoon Lee, Department of Mathematics, The Chinese University of Hong Kong, Shatin, N.T., Hong Kong}
\email{jaehoonlee@cuhk.edu.hk}

\author[E. Yeon]{Eungbeom Yeon}
\address[]{Eungbeom Yeon, Department of Mathematical Sciences, Pusan National University, Busan 46241, Republic of Korea}
\email{ebeom.yeon@pusan.ac.kr}

\allowdisplaybreaks

\begin{document}

\newtheorem{theorem}{theorem}[section]
\newtheorem{thm}[theorem]{Theorem}
\newtheorem{lemma}[theorem]{Lemma}
\newtheorem{cor}[theorem]{Corollary}
\newtheorem{prop}[theorem]{Proposition}
\newtheorem{rmk}[theorem]{Remark}
\newtheorem{Ex}[theorem]{Example}
\newtheorem{Question}[theorem]{Question}
\newtheorem{conj}[theorem]{Conjecture}
\newtheorem*{mainthm1}{Theorem 1}
\newtheorem*{mainthm2}{Theorem 2}

\renewcommand{\theequation}{\thesection.\arabic{equation}}
\newcommand{\RNum}[1]{\uppercase\expandafter{\romannumeral #1\relax}}
\newcommand{\R}{\mathbb{R}}
\newcommand{\C}{\mathbb{C}}
\newcommand{\grad}{\nabla}
\newcommand{\laplacian}{\Delta}
\newcommand{\tgamma}{\tilde{\gamma}}
\newcommand{\ttau}{\tilde{\tau}}
\newcommand{\pp}{\Phi}
\newcommand{\tkappa}{\tilde{\kappa}}
\renewcommand{\d}{\textup{d}}
\newlength{\mywidth}
\newcommand\bigfrown[2][\textstyle]{\ensuremath{%
  \array[b]{c}\text{\resizebox{\mywidth}{.7ex}{$#1\frown$}}\\[-1.3ex]#1#2\endarray}}
\newcommand{\arc}[1]{{%
  \setbox9=\hbox{#1}%
  \ooalign{\resizebox{\wd9}{\height}{\texttoptiebar{\phantom{A}}}\cr#1}}}
\newcommand\buildcirclepm[1]{%

  \begin{tikzpicture}[baseline=(X.base), inner sep=-0.1, outer sep=-2]

    \node[draw,circle] (X)  {\footnotesize\raisebox{1ex}{$#1\pm$}};

  \end{tikzpicture}%

}

\subjclass[2020]{53A10, 53C42}
\keywords{Costa--Hoffman--Meeks surface, Lagrangian catenoid, minimal surface, generalized Weierstrass representation, high codimension}

\begin{abstract}
In this paper, we study complete minimal surfaces in $\mathbb{R}^4$ with three embedded planar ends parallel to those of the union of the Lagrangian catenoid and the plane containing its waist circle. We show that any complete, oriented, immersed minimal surface in $\mathbb{R}^4$ with finite total curvature and genus $1$, and having these three ends, must be $J$-holomorphic for some almost complex structure $J$. If the surface is also embedded and has at least $8$ symmetries, we prove that its symmetry group has order $8$ or $12$. In each case the surface is uniquely determined up to scaling and rigid motion.
\end{abstract}

\maketitle

\section{Introduction}\label{intro}
\setcounter{equation}{0}
The discovery of Costa's surface marked a major breakthrough in minimal surface theory, as it provided the first known example of a complete embedded minimal surface in $\mathbb{R}^3$ with finite total curvature and positive genus. Costa \cite{C} constructed a genus-one surface with three embedded ends, each asymptotic to either a catenoid or a plane. Later, Hoffman and Meeks \cite{HM2} proved that Costa's surface is embedded and extended the construction to higher genus.

Under a suitable scaling, it was shown in \cite{HMARMA} that the Costa--Hoffman--Meeks surfaces converge to the union of the catenoid and the plane containing its waist circle as the genus tends to infinity. Moreover, it was proved in \cite{HMARMA} that an appropriate blow-up limit is given by Scherk's fifth minimal surface. Hence, the Costa--Hoffman--Meeks surfaces may be viewed as desingularizations of the catenoid and the plane containing its waist circle. They also provide important examples supporting Morgan's principle, which asserts that two transversely intersecting minimal surfaces can be desingularized using Scherk's family.

In higher codimension, there is a wider variety of catenoidal minimal surfaces. Hoffman and Osserman \cite{HO} classified complete doubly connected minimal surfaces in $\mathbb{R}^n$ with total curvature $-4\pi$. Among these, the surfaces foliated by circles are 
\begin{itemize}
\item[-] the catenoids in $\mathbb{R}^3$,
\item[-] the Lagrangian catenoids in $\mathbb{R}^4$,
\item[-] the Hoffman--Osserman catenoids in $\mathbb{R}^5$.
\end{itemize}
Here, by the Hoffman--Osserman catenoids, we mean the surfaces corresponding to Case $1$ of \cite[Proposition 6.6]{HO} for which $d_4$ and $d_5$ are not both zero and satisfy $d_4^2+d_5^2=0$. It is therefore natural to consider Costa--Hoffman--Meeks-type desingularizations for these catenoidal families. In this paper, we focus on this problem for the Lagrangian catenoid in $\mathbb{R}^4$.

Up to rigid motion and scaling, the Lagrangian catenoid can be identified with
\begin{align*}
\Sigma_{LC}:=\left\{(z, w)\in\mathbb{C}^2\ |\ zw=1\right\}.
\end{align*}
Since it is the graph of a holomorphic function, $\Sigma_{LC}$ is an embedded minimal surface. It is asymptotic to the two real $2$-planes $\{z=0\}$ and $\{w=0\}$. Moreover, writing $z=re^{i\theta}$, we have
\begin{align*}
\left(z,\frac{1}{z}\right)=\left(re^{i\theta},\frac{1}{r}e^{-i\theta}\right)=\cos\theta\left(r,0,\frac{1}{r},0\right)+\sin\theta\left(0,r,0,-\frac{1}{r}\right).
\end{align*}
Here, we identify $\mathbb{C}^2$ with $\mathbb{R}^4$ via
\begin{align*}
(z, w)\in\mathbb{C}^2\leftrightarrow (\text{Re}z, \text{Im}z, \text{Re}w, \text{Im}w)\in\mathbb{R}^4
\end{align*}
and use this identification throughout the paper. The above representation shows that $\Sigma_{LC}$ is foliated by circles of radius $\sqrt{r^2+\frac{1}{r^2}}$, with the circle corresponding to $r=1$ being the waist circle. We consider a possible desingularization of $\Sigma_{LC}$ and the plane containing its waist circle. 

However, a variational approach faces significant difficulties in higher codimension. Instead, inspired by the original construction of the Costa--Hoffman--Meeks surfaces via the Weierstrass representation, we approach the problem using its higher-codimensional generalization. This generalized Weierstrass representation was extensively studied in \cite{HO}.

If such a desingularization exists, it must have three ends, each asymptotic to a plane. A key difference between codimension one and higher codimension is that, in higher codimension, even an embedded end may approach its asymptotic plane with multiplicity greater than one. We restrict our attention to multiplicity-one ends, as in the Lagrangian catenoid, and refer to such ends as embedded planar ends. In \cite{JL}, it was shown that the corresponding Weierstrass data satisfy the same characterization as for embedded planar ends in the classical setting of $\mathbb{R}^3$. See Proposition \ref{PropOrdRes} for details.

In the genus-one case, we obtain the following $J$-holomorphicity result.
\begin{mainthm1}[Theorem \ref{thm1}]\label{mainthm1}
Let $S_1$ be a complete, oriented, immersed minimal surface in $\mathbb{R}^4$ with finite total curvature and genus $1$. Suppose that $S_1$ has three embedded planar ends whose asymptotic planes are parallel to those of the union of $\Sigma_{LC}$ and the plane containing its waist circle. Then $S_1$ must be a $J$-holomorphic curve for some almost complex structure $J$.
\end{mainthm1}

For the proof, we first identify the underlying Riemann surface with $\mathbb{C}/\Lambda(1,\tau)$, where $\Lambda(1,\tau)$ is the lattice generated by $1$ and $\tau$. We then express the Weierstrass data in terms of the generalized Gauss map at the three ends. Substituting these data into the conformality condition that the sum of their squares vanishes, we find two possible families: one corresponding to $J$-holomorphic surfaces and the other to non-$J$-holomorphic surfaces. For the latter family, we compute the periods explicitly and, using several identities for elliptic functions, show that they necessarily have a nonzero real period.

The Costa--Hoffman--Meeks surface of genus $g$ in $\mathbb{R}^3$ has $4(g+1)$ symmetries and is uniquely characterized among complete embedded minimal surfaces of finite total curvature and genus $g$ with three ends by having at least $4(g+1)$ symmetries (see \cite{HM2}). In the genus-one case, we also obtain the following uniqueness result.

\begin{mainthm2}[Proposition \ref{proptwocase}; Theorems \ref{thm2} and \ref{thm3}]
Let $S_1\subset\mathbb{R}^4$ be a complete, oriented, embedded minimal surface with finite total curvature and genus $1$, having three embedded planar ends whose asymptotic planes are parallel to those of the union of $\Sigma_{LC}$ and the plane containing its waist circle. Suppose that $S_1$ has at least $8$ symmetries, that is, $\left|\mathfrak{Sym}(S_1)\right|\geq 8$. Then either $\left|\mathfrak{Sym}(S_1)\right|=8$ or $\left|\mathfrak{Sym}(S_1)\right|=12$. In either case, $S_1$ is uniquely determined up to scaling and rigid motion. 
\end{mainthm2}

We remark that, in the case of $8$ symmetries, the underlying Riemann surface is conformally equivalent to that of the Costa surface. In the case of $12$ symmetries, the surface is given by
\begin{align*}
\left\{ (z, w) \in \mathbb{C}^2 \ \middle|\ z w (z - w) = 1 \right\}
\end{align*}
up to scaling and rigid motion. Its defining equation is similar in form to that of the Lagrangian catenoid $\Sigma_{LC}$.

The idea of the proof is as follows. Motivated by \cite{HM2}, we consider the subgroup $\mathcal{H}$ of orientation-preserving symmetries that preserve the three ends. Under the assumption that $\left|\mathfrak{Sym}(S_1)\right|\geq 8$, we show that $|\mathcal{H}|=2$ or $3$, and that $\left|\mathfrak{Sym}(S_1)\right|=4|\mathcal{H}|$ (see Proposition \ref{proptwocase}). By Theorem \ref{mainthm1}, $S_1$ is represented by two meromorphic functions. In each case, we analyze the orders of their zeros and poles and show that these functions are uniquely determined up to multiplication by nonzero constants.

When the genus is greater than one, the picture is substantially different from the genus-one case. In particular, no such desingularization can have at least $4(g+1)$ symmetries. We address this higher-genus case in \cite{JLEY}.

We may also consider the desingularization problem for the Hoffman--Osserman catenoids in $\mathbb{R}^5$. Using the generalized Weierstrass representation, for each genus $g$, we construct a new complete embedded minimal surface with three ends, lying entirely in $\mathbb{R}^5$, and having the same intrinsic symmetries as the Costa--Hoffman--Meeks surface. This result will appear in a forthcoming paper.

The paper is organized as follows. Section \ref{SecSec2} recalls the generalized Weierstrass representation and the characterization of embedded planar ends in terms of the Weierstrass data. There, we also identify the underlying Riemann surface with $\mathbb{C}/\Lambda(1,\tau)$ and express the Weierstrass data in terms of the values of the generalized Gauss map at each end. In Section \ref{sec3}, we prove the $J$-holomorphicity result, Theorem \ref{mainthm1}. Section \ref{SecSec4} is devoted to the structure of the symmetry group; in particular, we prove that if the surface is embedded and has at least $8$ symmetries, then it has either $8$ or $12$ symmetries. Finally, in Section \ref{sec5}, we treat the cases of $8$ and $12$ symmetries separately and show that, in each case, the surface is uniquely determined up to scaling and rigid motion.

\section*{Acknowledgements}
JL was supported by a KIAS Individual Grant MG086402 at Korea Institute for Advanced Study, in part by National Research Foundation of Korea [No. RS-2026-25472185], and substantially by research grants from the Research Grants Council of the Hong Kong Special Administrative Region, China [Project No.: CUHK 14306624]. EY was supported by National Research Foundation of Korea NRF-2022R1C1C2013384 and in part by NRF-2021R1A4A1032418. Parts of this work were carried out during visits to the University of Ja\'{e}n and Stanford University. The authors would like to express their sincere gratitude to Professor Ildefonso Castro at the University of Ja\'{e}n and Professor Otis Chodosh at Stanford University for their kind invitations. The authors also thank both institutions for their hospitality and support.

\section{Weierstrass Representation for Planar Ends}\label{SecSec2}
\setcounter{equation}{0}

\subsection{Embedded planar ends in high codimension}
By the Chern--Osserman theorem, a complete immersed minimal surface in $\mathbb{R}^N$ with finite total curvature and genus $g$ is conformally equivalent to a closed Riemann surface of the same genus with finitely many punctures. Hence, such a surface can be represented by an immersion $X: \Sigma_g\setminus\{q_1,q_2,\cdots,q_k\}\to\mathbb{R}^N$,
\begin{align}\label{gwr}
X=\text{Re}\int (\phi_1,\phi_2,\cdots,\phi_N)+b
\end{align}
for some $b\in\mathbb{R}^N$, where $\Sigma_g$ is a closed Riemann surface of genus $g$. The $\phi_j$ are meromorphic $1$-forms on $\Sigma_g$ satisfying the following conditions:
\begin{itemize}
\item[(1)] Each $\phi_j$ may have poles only at $\{q_1, q_2, \cdots,q_k\}$, and at least one of them has a pole at each $q_j$.
\item[(2)] $\sum_{j=1}^N\phi_j^2\equiv0$ and $\sum_{j=1}^N|\phi_j|^2>0$.
\item[(3)] All real periods of the $\phi_j$ vanish.
\end{itemize}
Conversely, a collection of meromorphic $1$-forms satisfying the above conditions determines, through the formula (\ref{gwr}), a conformal harmonic immersion and hence a complete minimal surface with finite total curvature and genus $g$ in $\mathbb{R}^N$. In this setting, each puncture $q_j$ corresponds to an end of the surface. 

The formula (\ref{gwr}) is called the \emph{generalized Weierstrass representation}, and the meromorphic $1$-forms $\phi_j$ are called the (generalized) \emph{Weierstrass data}. The associated map 
\begin{align*}
\Phi:=[\phi_1,\phi_2,\cdots,\phi_N]:\Sigma_g\to\mathbb{P}^{N-1}
\end{align*}
is called the \emph{generalized Gauss map}. 

A minimal immersion $X$ is said to have an \textbf{embedded planar end at $q_j$} if there exists an open neighborhood $U_j$ of $q_j$ such that $X(U_j\setminus\{q_j\})$ is asymptotic to a plane and can be represented as a graph over that plane outside a compact subset. Moreover, none of the height functions has logarithmic growth. We say that the immersion $X$ has embedded planar ends if it has an embedded planar end at each $q_j$.

This definition describes an end asymptotic to a plane with multiplicity one, as in the case of the Lagrangian catenoid. It is also consistent with the classical notion of an embedded planar end in $\mathbb{R}^3$. The following equivalent characterization was proved in \cite{JL}.
\begin{prop}[Proposition 2.1 in \cite{JL}]\label{PropOrdRes}
A minimal immersion $X$ given by the generalized Weierstrass representation (\ref{gwr}) has embedded planar ends if and only if
\begin{align}\label{OrdRes}
\min_{1\leq j\leq N}\textup{ord}_{q_i}\phi_j=-2\quad\text{and}\quad\textup{res}_{q_i}\phi_1=\textup{res}_{q_i}\phi_2=\cdots=\textup{res}_{q_i}\phi_N=0
\end{align} 
for all $i=1,2,\cdots,k$.
\end{prop}

Let $K$ denote the canonical line bundle of $\Sigma_g$. For each puncture $q_j$, choose 
\begin{align}\label{ETA}
\eta_j\in H^0(\Sigma_g, K\otimes [2q_j])\setminus H^0(\Sigma_g,K).
\end{align}
By the Riemann--Roch formula and the residue formula, any meromorphic $1$-form $\phi$ satisfying (\ref{OrdRes}) belongs to 
\begin{align}\label{decompdecomp}
H^0(\Sigma_g, K)\oplus\left(\bigoplus_{1\leq j\leq k}\mathbb{C}\eta_j\right).
\end{align}
See \cite[Lemma 3.4]{JLEY} for further details. In the next subsection, we use this fact together with the generalized Gauss map to express the Weierstrass data of minimal surfaces with embedded planar ends.

\subsection{Weierstrass $\wp$-function and data}\label{subsec2222}
Suppose that a minimal immersion $X$ given by (\ref{gwr}) has three embedded planar ends in $\mathbb{R}^4$. When $g=1$, the underlying Riemann surface $\Sigma_1$ can be identified with 
\begin{align*}
\mathbb{C}/\Lambda(1,\tau),
\end{align*}
where
\begin{align*} 
\Lambda(1,\tau) := \{m+n\tau\ |\ m,n\in \mathbb{Z}\}
\end{align*} 
for some 
\begin{align*}
\tau\in\left\{x+iy\in\mathbb{C}\ \bigg{|}\ x^2+y^2\geq 1,\ y>0,\ -\frac{1}{2}\leq x\leq \frac{1}{2}\right\}=:F.
\end{align*}
The set $F$ is a fundamental domain for the moduli space of closed Riemann surfaces of genus $1$.

Let $\wp(z;\tau)$ be the Weierstrass $\wp$-function on $\mathbb{C}/\Lambda(1,\tau)$, satisfying the differential equation
\begin{align}\label{pftn}
(\wp^\prime)^2 = 4(\wp-e_1)(\wp-e_2)(\wp-e_3),
\end{align}
where 
\begin{align*}
e_1 = \wp\left(\frac{1}{2};\tau\right),\quad e_2 = \wp\left(\frac{\tau}{2};\tau\right),\quad e_3 = \wp\left(\frac{1+\tau}{2};\tau\right) 
\end{align*}
are the half-period values of $\wp(z;\tau)$. We write $\wp(z)$ for $\wp(z;\tau)$ when $\tau$ is clear from context.

Using the Weierstrass $\wp$-function, a meromorphic $1$-form as in (\ref{ETA}) can be chosen explicitly as follows.
\begin{lemma} \label{etacase}
A meromorphic $1$-form $\eta$ on $\Sigma_1=\mathbb{C}/\Lambda(1,\tau)$ satisfying
\begin{align*}
\eta\in H^0(\Sigma_1, K\otimes [2q])\setminus H^0(\Sigma_1,K),\quad q\in\Sigma_1,
\end{align*}
can be chosen as 
\begin{align*}
\eta = \begin{cases}
\frac{\wp^\prime(z) + \wp^\prime (q) + \frac{\wp^{\prime\prime} (q)}{\wp^{\prime}(q)} (\wp(z) -\wp(q))}{(\wp(z) - \wp(q))^2} dz,&\text{if}  \quad  q \notin \left\{0, \frac{1}{2}, \frac{\tau}{2}, \frac{1+\tau}{2}\right\},\\
\frac{1}{\wp(z)-\wp(q)} dz,&\text{if} \quad  q \in \left\{ \frac{1}{2}, \frac{\tau}{2}, \frac{1+\tau}{2}\right\},\\
\wp(z)dz,&\text{if}\quad q=0.
\end{cases} 
\end{align*}
\end{lemma}
\begin{proof}
See \cite[Lemma 3.3]{JL} for the proof. 
\end{proof}

Let $\eta_1$, $\eta_2$, $\eta_3$ as in (\ref{ETA}) be chosen according to the previous lemma. Since the holomorphic $1$-forms on $\mathbb{C}/\Lambda(1,\tau)$ are precisely the complex multiples of $dz$, it follows from (\ref{decompdecomp}) that any meromorphic $1$-form $\phi$ satisfying (\ref{OrdRes}) can be written as
\begin{align*}
\phi=l_0dz+l_1\eta_1+l_2\eta_2+l_3\eta_3
\end{align*}
for some complex numbers $l_0,l_1,l_2,l_3$. Since only $\eta_j$ has a pole at $q_j$, the coefficients of $\eta_j$ in the Weierstrass data determine the generalized Gauss map at $q_j$. Thus, we have the following proposition.
\begin{prop}\label{prop27}
Let $X:\Sigma_1\setminus\{q_1,q_2,q_3\}\to\mathbb{R}^4$ be a complete minimal immersion with finite total curvature and genus $1$, represented by (\ref{gwr}). Identify $\Sigma_1$ with $\mathbb{C}/\Lambda(1,\tau)$ for some $\tau\in F$. Suppose that $X$ has embedded planar ends, and let $\eta_1, \eta_2, \eta_3$ as in (\ref{ETA}) be chosen according to Lemma \ref{etacase}. If the generalized Gauss map at each end $q_j$ is given by
\begin{align*}
\Phi(q_j)=[a_j, b_j, c_j, d_j]\in\mathbb{P}^3,
\end{align*}
then there exist nonzero complex numbers $\alpha, \beta, \gamma$ and complex numbers $a_0, b_0, c_0, d_0$ such that the Weierstrass data take the form
\begin{align*}
&\phi_1=a_0dz+\alpha a_1\eta_1+\beta a_2\eta_2+\gamma a_3\eta_3,\\
&\phi_2=b_0dz+\alpha b_1\eta_1+\beta b_2\eta_2+\gamma b_3\eta_3,\\
&\phi_3=c_0dz+\alpha c_1\eta_1+\beta c_2\eta_2+\gamma c_3\eta_3,\\
&\phi_4=d_0dz+\alpha d_1\eta_1+\beta d_2\eta_2+\gamma d_3\eta_3.
\end{align*}
\end{prop}


\section{Holomorphicity of Immersed Minimal Tori}\label{sec3}
\setcounter{equation}{0}
In this section, we prove the following theorem.
\begin{thm}\label{thm1}
Let $S_1 $ be a complete, oriented, immersed minimal surface in $\mathbb{R}^4$ with finite total curvature and genus $1$. Suppose that $S_1$ has three embedded planar ends whose asymptotic planes are parallel to those of the union of the Lagrangian catenoid $\Sigma_{LC}$ and the plane containing its waist circle. Then $S_1$ must be a $J$-holomorphic curve for some almost complex structure $J$.
\end{thm}

The proof proceeds in three steps. In Subsection \ref{etaweierstrass}, we first determine the possible Weierstrass data and analyze the identity $\sum_{j=1}^4 \phi_j^2\equiv 0$. This reduces the proof to either the $J$-holomorphic case or a single remaining non-holomorphic case. In Subsection \ref{periodcompute}, we study the real period conditions for the latter case and reformulate them in terms of the rank of the matrix $M_{\tau, \wp_1,\wp_3}$. Finally, in the last subsection, we use properties of elliptic functions and the $j$-invariant to show that the required rank condition cannot occur, thereby ruling out the non-holomorphic case and completing the proof.

\subsection{The Weierstrass data}\label{etaweierstrass}
Let us denote the asymptotic planes of the Lagrangian catenoid $\Sigma_{LC}$ by
\begin{align*}
Q_1:= \left\{(z, 0)\in\mathbb{C}^2\ |\ z\in\mathbb{C}\right\},\quad Q_3:=\left\{(0,z)\in\mathbb{C}^2\ |\ z\in\mathbb{C}\right\},
\end{align*}
and the plane containing its waist circle by
\begin{align*}
Q_2:=\left\{\left(z, \bar{z}\right)\in\mathbb{C}^2\ |\ z\in\mathbb{C}\right\}.
\end{align*}

Suppose that $S_1$ is represented by an immersion $X:\Sigma_1\setminus\{q_1,q_2,q_3\}\to\mathbb{R}^4$ via the generalized Weierstrass representation (\ref{gwr}), where each puncture $q_j$ corresponds to an embedded planar end of $S_1$ parallel to $Q_j$. As discussed in Section \ref{subsec2222}, we may identify $\Sigma_1$ with 
\begin{align*}
\mathbb{C}/\Lambda(1,\tau)
\end{align*}
for some $\tau\in F$. After a translation on $\mathbb{C}/\Lambda(1,\tau)$, we may further assume that 
\begin{align*}
q_2=\overline{0}\in\mathbb{C}/\Lambda(1,\tau).
\end{align*}

Since the asymptotic planes of $S_1$ are parallel to $Q_1, Q_2, Q_3$, the values of the generalized Gauss map at $q_1,q_2,q_3$ can be written as
\begin{align*}
\Phi(q_1)=[1,\sigma_1 i, 0, 0]\in\mathbb{P}^3,\quad \Phi(q_2)=[1,\sigma_2 i, 1, -\sigma_2 i]\in\mathbb{P}^3,
\end{align*}
and
\begin{align*}
\Phi(q_3)=[0,0, 1, -\sigma_3 i]\in\mathbb{P}^3
\end{align*}
for some $\sigma_1, \sigma_2, \sigma_3\in \{-1,1\}$. By reversing the orientation of $S_1$ if necessary, we may assume that $\sigma_1=1$. Thus,
\begin{align*}
\Phi(q_1)=[1,i, 0, 0]\in\mathbb{P}^3.
\end{align*}
The orientation of $S_1$ then determines the remaining signs $\sigma_2$ and $\sigma_3$.

Choose meromorphic $1$-forms $\eta_1,\eta_2,\eta_3$ satisfying (\ref{ETA}) according to Lemma \ref{etacase}. By Proposition \ref{prop27}, the Weierstrass data associated with $S_1$ can be written as
\begin{equation}\label{vvv1}
\left\{
\begin{alignedat}{4}
\phi_1 &= a_0dz &{}+{}& \alpha\eta_1
&{}+{}& \beta\eta_2
&{}+{}& 0\\
\phi_2 &= b_0dz &{}+{}& i\alpha\eta_1
&{}+{}& i\sigma_2\beta\eta_2
&{}+{}& 0\\
\phi_3 &= c_0dz &{}+{}& 0
&{}+{}& \beta\eta_2
&{}+{}& \gamma\eta_3\\
\phi_4 &= d_0dz &{}+{}& 0
&{}-{}& i\sigma_2\beta\eta_2
&{}-{}& i\sigma_3\gamma\eta_3
\end{alignedat}
\right.
\end{equation}
for some complex numbers $a_0,b_0,c_0,d_0$ and nonzero complex numbers $\alpha,\beta,\gamma$. 

Since we have identified $q_2$ with $\overline{0}\in\mathbb{C}/\Lambda(1,\tau)$, the choice of $\eta_2$ in Lemma \ref{etacase} becomes
\begin{align*}
\eta_2=\wp(z)dz.
\end{align*}
Let $I_{H} := \{\frac{1}{2}, \frac{\tau}{2}, \frac{1+\tau}{2} \}$. By Lemma \ref{etacase}, there are four possible cases for the pair $(\eta_1, \eta_3)$, depending on whether $q_1$ and $q_3$ belong to $I_H$. For $j=1,3$, define
\begin{align*}
D_j(z):=\frac{\wp^\prime(z) + \wp^\prime (q_j) + \frac{\wp^{\prime\prime} (q_j)}{\wp^{\prime}(q_j)} (\wp(z) -\wp(q_j))}{(\wp(z) - \wp(q_j))^2}
\end{align*}
when $q_j\not\in I_H$, and 
\begin{align*}
H_j(z):=\frac{1}{\wp(z)-\wp(q_j)}
\end{align*}
when $q_j\in I_H$. Then
\begin{align*}
\eta_j=
\begin{cases}
D_j(z)dz,\quad\text{if}\quad q_j\not\in I_H,\\
H_j(z)dz,\quad\text{if}\quad q_j\in I_H.
\end{cases}
\end{align*}

In the following lemmas, we narrow down the possible Weierstrass data for such surfaces, if they exist, using algebraic consequences of the identity $\sum_{j=1}^4\phi_j^2\equiv0$. 

\begin{lemma}\label{etacase1}
Assume that one of the following holds:
\begin{align*}
\textup{(a)}\ q_1\notin I_H,\ q_3\notin I_H,\quad \textup{(b)}\ q_1\in I_H,\ q_3\notin I_H,\quad \textup{(c)}\ q_1\notin I_H,\ q_3\in I_H.
\end{align*}
Then the Weierstrass data in (\ref{vvv1}) satisfy the identity $\sum_{j=1}^4\phi_j^2\equiv0$ if and only if 
\begin{align*}
\phi_1 \equiv -i \phi_2,\quad \phi_3 \equiv i \phi_4. 
\end{align*}
In this case, the corresponding minimal surface is $J$-holomorphic for some almost complex structure $J$.
\end{lemma}

\begin{proof}
For clarity, we list the meromorphic $1$-forms $\eta_1$ and $\eta_3$ in the three cases as follows:
\begin{align*}
(\eta_1, \eta_3) =
\begin{cases}
(D_1(z) \, dz,\; D_3(z) \, dz)  \cdots &\textup{{Case (a)}} \\
(H_1(z) \, dz,\; D_3(z) \, dz)\cdots &\textup{{Case (b)}} \\
(D_1(z) \, dz,\; H_3(z) \, dz)\cdots &\textup{{Case (c)}} 
\end{cases}.
\end{align*}
The \lq\lq if\rq\rq\ direction is immediate. Thus, we assume that $\sum_{j=1}^4\phi_j^2\equiv0$ and prove that $\phi_1 \equiv -i \phi_2$ and $\phi_3 \equiv i \phi_4$. 
\bigskip

\paragraph{\textbf{\underline{Case (a)}}} When $\eta_1=D_1(z)dz$ and $\eta_3=D_3(z)dz$, we have
\begin{align*}
0&=\left(a_0+\alpha D_1(z)+\beta\wp(z)\right)^2+\left(b_0+i\alpha D_1(z)+i\sigma_2\beta\wp(z)\right)^2\\
&\qquad+\left(c_0+\beta\wp(z)+\gamma D_3(z)\right)^2+\left(d_0-i\sigma_2\beta\wp(z)-i\sigma_3\gamma D_3(z)\right)^2\\
&=\left(a_0^2+b_0^2+c_0^2+d_0^2\right)+2\alpha(a_0+ib_0)\frac{\wp'(z)+\wp'(q_1)+\frac{\wp''(q_1)}{\wp'(q_1)}\left(\wp(z)-\wp(q_1)\right)}{\left(\wp(z) - \wp(q_1)\right)^2}\\
&\qquad+2\beta(a_0+i\sigma_2b_0+c_0-i\sigma_2d_0)\wp(z)\\
&\qquad+2\gamma(c_0-i\sigma_3d_0)\frac{\wp'(z)+\wp'(q_3)+\frac{\wp''(q_3)}{\wp'(q_3)}\left(\wp(z)-\wp(q_3)\right)}{\left(\wp(z)-\wp(q_3)\right)^2}\\
&\qquad+2\alpha\beta(1-\sigma_2)\wp(z)\frac{\wp'(z)+\wp'(q_1)+\frac{\wp''(q_1)}{\wp'(q_1)}\left(\wp(z)-\wp(q_1)\right)}{\left(\wp(z) - \wp(q_1)\right)^2}\\
&\qquad+2\beta\gamma(1-\sigma_2\sigma_3)\wp(z)\frac{\wp'(z)+\wp'(q_3)+\frac{\wp''(q_3)}{\wp'(q_3)}\left(\wp(z)-\wp(q_3)\right)}{\left(\wp(z)-\wp(q_3)\right)^2}
\end{align*}
from the identity $\sum_{j=1}^4\phi_j^2\equiv0$. Multiplying by $\left(\wp(z)-\wp(q_1)\right)^2\left(\wp(z)-\wp(q_3)\right)^2$ and isolating the terms involving $\wp'(z)$, we obtain
\begin{align*}
&-2\wp'(z)\left[\alpha(a_0+ib_0)\left(\wp(z)-\wp(q_3)\right)^2+\gamma(c_0-i\sigma_3d_0)\left(\wp(z)-\wp(q_1)\right)^2\right.\\
&\qquad\left.+\alpha\beta(1-\sigma_2)\wp(z)\left(\wp(z)-\wp(q_3)\right)^2+\beta\gamma(1-\sigma_2\sigma_3)\wp(z)\left(\wp(z)-\wp(q_1)\right)^2\right]\\
&=\left(a_0^2+b_0^2+c_0^2+d_0^2\right)\left(\wp(z)-\wp(q_1)\right)^2\left(\wp(z)-\wp(q_3)\right)^2\\
&\qquad+2\alpha(a_0+ib_0)\left(\wp'(q_1)+\frac{\wp''(q_1)}{\wp'(q_1)}\left(\wp(z)-\wp(q_1)\right)\right)\left(\wp(z)-\wp(q_3)\right)^2\\
&\qquad+2\beta(a_0+i\sigma_2b_0+c_0-i\sigma_2d_0)\wp(z)\left(\wp(z)-\wp(q_1)\right)^2\left(\wp(z)-\wp(q_3)\right)^2\\
&\qquad+2\gamma(c_0-i\sigma_3d_0)\left(\wp'(q_3)+\frac{\wp''(q_3)}{\wp'(q_3)}\left(\wp(z)-\wp(q_3)\right)\right)\left(\wp(z)-\wp(q_1)\right)^2\\
&\qquad+2\alpha\beta(1-\sigma_2)\wp(z)\left(\wp'(q_1)+\frac{\wp''(q_1)}{\wp'(q_1)}\left(\wp(z)-\wp(q_1)\right)\right)\left(\wp(z)-\wp(q_3)\right)^2\\
&\qquad+2\beta\gamma(1-\sigma_2\sigma_3)\wp(z)\left(\wp'(q_3)+\frac{\wp''(q_3)}{\wp'(q_3)}\left(\wp(z)-\wp(q_3)\right)\right)\left(\wp(z)-\wp(q_1)\right)^2.
\end{align*}
Squaring both sides and applying (\ref{pftn}), it follows that
\begin{align*}
&16\left(\wp(z)-e_1\right)\left(\wp(z)-e_2\right)\left(\wp(z)-e_3\right)\\
&\times\left[\alpha(a_0+ib_0)\left(\wp(z)-\wp(q_3)\right)^2+\gamma(c_0-i\sigma_3d_0)\left(\wp(z)-\wp(q_1)\right)^2\right.\\
&\qquad\left.+\alpha\beta(1-\sigma_2)\wp(z)\left(\wp(z)-\wp(q_3)\right)^2+\beta\gamma(1-\sigma_2\sigma_3)\wp(z)\left(\wp(z)-\wp(q_1)\right)^2\right]^2\\
&=\biggl[\left(a_0^2+b_0^2+c_0^2+d_0^2\right)\left(\wp(z)-\wp(q_1)\right)^2\left(\wp(z)-\wp(q_3)\right)^2\\
&\qquad+2\alpha(a_0+ib_0)\left(\wp'(q_1)+\frac{\wp''(q_1)}{\wp'(q_1)}\left(\wp(z)-\wp(q_1)\right)\right)\left(\wp(z)-\wp(q_3)\right)^2\\
&\qquad+2\beta(a_0+i\sigma_2b_0+c_0-i\sigma_2d_0)\wp(z)\left(\wp(z)-\wp(q_1)\right)^2\left(\wp(z)-\wp(q_3)\right)^2\\
&\qquad+2\gamma(c_0-i\sigma_3d_0)\left(\wp'(q_3)+\frac{\wp''(q_3)}{\wp'(q_3)}\left(\wp(z)-\wp(q_3)\right)\right)\left(\wp(z)-\wp(q_1)\right)^2\\
&\qquad+2\alpha\beta(1-\sigma_2)\wp(z)\left(\wp'(q_1)+\frac{\wp''(q_1)}{\wp'(q_1)}\left(\wp(z)-\wp(q_1)\right)\right)\left(\wp(z)-\wp(q_3)\right)^2\\
&\qquad +2\beta\gamma(1-\sigma_2\sigma_3)\wp(z)\left(\wp'(q_3)+\frac{\wp''(q_3)}{\wp'(q_3)}\left(\wp(z)-\wp(q_3)\right)\right)\left(\wp(z)-\wp(q_1)\right)^2\biggr]^2.
\end{align*}
Since the Weierstrass $\wp$-function takes infinitely many values, both sides must agree as polynomials in $\wp(z)$. Comparing their degrees and noting that the right-hand side is a square of a polynomial in $\wp(z)$, we conclude that both sides must vanish identically. 

In particular, we have
\begin{align}\label{RRRRRR1}
&\alpha(a_0+ib_0)\left(\wp(z)-\wp(q_3)\right)^2+\gamma(c_0-i\sigma_3d_0)\left(\wp(z)-\wp(q_1)\right)^2\nonumber\\
&+\alpha\beta(1-\sigma_2)\wp(z)\left(\wp(z)-\wp(q_3)\right)^2+\beta\gamma(1-\sigma_2\sigma_3)\wp(z)\left(\wp(z)-\wp(q_1)\right)^2\equiv0
\end{align}
from the left-hand side. Evaluating (\ref{RRRRRR1}) at $z=q_1$ and $z=q_3$, and using the fact that $\wp(q_1)\neq\wp(q_3)$, we obtain
\begin{align}\label{RRRRRR2}
\alpha(a_0+ib_0)+\alpha\beta(1-\sigma_2)\wp(q_1)=0
\end{align}
and
\begin{align}\label{RRRRRR3}
\gamma(c_0-i\sigma_3d_0)+\beta\gamma(1-\sigma_2\sigma_3)\wp(q_3)=0.
\end{align}
Substituting (\ref{RRRRRR2}) and (\ref{RRRRRR3}) into (\ref{RRRRRR1}) yields
\begin{align*}
\alpha\beta(1-\sigma_2)\left(\wp(z)-\wp(q_1)\right)&\left(\wp(z)-\wp(q_3)\right)^2\\
&+\beta\gamma(1-\sigma_2\sigma_3)\left(\wp(z)-\wp(q_1)\right)^2\left(\wp(z)-\wp(q_3)\right)\equiv0.
\end{align*}
Since $\wp(q_1)\neq\wp(q_3)$, we conclude that
\begin{align*}
\alpha\beta(1-\sigma_2)=\beta\gamma(1-\sigma_2\sigma_3)=0.
\end{align*}
As $\alpha$, $\beta$, and $\gamma$ are all nonzero by assumption, it follows that $\sigma_2=\sigma_3=1$. Substituting these into (\ref{RRRRRR2}) and (\ref{RRRRRR3}), we find
\begin{align*}
a_0+ib_0=0,\quad c_0-id_0=0.
\end{align*}
Hence, from (\ref{vvv1}), we conclude that $\phi_1 \equiv -i \phi_2$ and $\phi_3 \equiv i \phi_4$, completing the proof for this case.
\bigskip

\paragraph{\textbf{\underline{Case (b)}}}
When $\eta_1=H_1(z)dz$ and $\eta_3=D_3(z)dz$, similar computations yield
\begin{align*}
&16\gamma^2\left(\wp(z)-e_1\right)\left(\wp(z)-e_2\right)\left(\wp(z)-e_3\right)\left(\wp(z)-\wp(q_1)\right)^2\left[c_0-i\sigma_3d_0+\beta(1-\sigma_2\sigma_3)\wp(z)\right]^2\\
&=\biggl[\left(a_0^2+b_0^2+c_0^2+d_0^2\right)\left(\wp(z)-\wp(q_1)\right)\left(\wp(z)-\wp(q_3)\right)^2\\
&\qquad+2\alpha(a_0+ib_0)\left(\wp(z)-\wp(q_3)\right)^2\\
&\qquad+2\beta(a_0+i\sigma_2b_0+c_0-i\sigma_2d_0)\wp(z)\left(\wp(z)-\wp(q_1)\right)\left(\wp(z)-\wp(q_3)\right)^2\\
&\qquad+2\gamma(c_0-i\sigma_3d_0)\left(\wp'(q_3)+\frac{\wp''(q_3)}{\wp'(q_3)}\left(\wp(z)-\wp(q_3)\right)\right)\left(\wp(z)-\wp(q_1)\right)\\
&\qquad+2\alpha\beta(1-\sigma_2)\wp(z)\left(\wp(z)-\wp(q_3)\right)^2\\
&\qquad+2\beta\gamma(1-\sigma_2\sigma_3)\wp(z)\left(\wp'(q_3)+\frac{\wp''(q_3)}{\wp'(q_3)}\left(\wp(z)-\wp(q_3)\right)\right)\left(\wp(z)-\wp(q_1)\right)\biggr]^2.
\end{align*}
As in the first case, both sides must agree as polynomials in $\wp(z)$. Since the degrees of both sides have different parity, and the right-hand side is a square of a polynomial, it follows that both sides must vanish identically.

From the left-hand side, we obtain the conditions
\begin{align*}
c_0-i\sigma_3d_0=0,\quad 1-\sigma_2\sigma_3=0.
\end{align*}
If $\sigma_2=\sigma_3=1$, then $c_0-id_0=0$. Substituting these into the right-hand side and using the fact that it vanishes identically, we find
\begin{align*}
\left(a_0^2+b_0^2\right)\left(\wp(z)-\wp(q_1)\right)+2\alpha(a_0+ib_0)+2\beta(a_0+ib_0)\wp(z)\left(\wp(z)-\wp(q_1)\right)\equiv0.
\end{align*}
This implies that $a_0+ib_0=0$. Hence, from (\ref{vvv1}), we conclude that $\phi_1 \equiv -i \phi_2$ and $\phi_3 \equiv i \phi_4$.

On the other hand, if $\sigma_2=\sigma_3=-1$, then $c_0+id_0=0$. Substituting these into the right-hand side and using the vanishing condition, we obtain
\begin{align*}
&\left(a_0^2+b_0^2\right)\left(\wp(z)-\wp(q_1)\right)+2\alpha(a_0+ib_0)\\
&\qquad+2\beta(a_0-ib_0)\wp(z)\left(\wp(z)-\wp(q_1)\right)+4\alpha\beta\wp(z)\equiv0.
\end{align*}
The coefficient of $\wp(z)^2$ must vanish, yielding $a_0-ib_0=0$. Substituting this back into the above identity gives
\begin{align*}
4\alpha a_0+4\alpha\beta\wp(z)\equiv0.
\end{align*}
Since $\alpha\beta\neq0$, this identity cannot hold, leading to a contradiction. This completes the proof for this case.
\bigskip

\paragraph{\textbf{\underline{Case (c)}}}
When $\eta_1=D_1(z) dz$ and $\eta_3=H_3(z) dz$, we follow the same procedure as before and obtain
\begin{align*}
&16\alpha^2\left(\wp(z)-e_1\right)\left(\wp(z)-e_2\right)\left(\wp(z)-e_3\right)\left(\wp(z)-\wp(q_3)\right)^2\left[a_0+ib_0+\beta(1-\sigma_2)\wp(z)\right]^2\\
&=\biggl[\left(a_0^2+b_0^2+c_0^2+d_0^2\right)\left(\wp(z)-\wp(q_1)\right)^2\left(\wp(z)-\wp(q_3)\right)\\
&\qquad+2\alpha(a_0+ib_0)\left(\wp'(q_1)+\frac{\wp''(q_1)}{\wp'(q_1)}\left(\wp(z)-\wp(q_1)\right)\right)\left(\wp(z)-\wp(q_3)\right)\\
&\qquad+2\beta(a_0+i\sigma_2b_0+c_0-i\sigma_2d_0)\wp(z)\left(\wp(z)-\wp(q_1)\right)^2\left(\wp(z)-\wp(q_3)\right)\\
&\qquad+2\gamma(c_0-i\sigma_3d_0)\left(\wp(z)-\wp(q_1)\right)^2\\
&\qquad+2\alpha\beta(1-\sigma_2)\wp(z)\left(\wp'(q_1)+\frac{\wp''(q_1)}{\wp'(q_1)}\left(\wp(z)-\wp(q_1)\right)\right)\left(\wp(z)-\wp(q_3)\right)\\
&\qquad+2\beta\gamma(1-\sigma_2\sigma_3)\wp(z)\left(\wp(z)-\wp(q_1)\right)^2\biggr]^2.
\end{align*}
As in the previous cases, both sides must be identically zero.

From the vanishing of the left-hand side, we obtain
\begin{align*}
\sigma_2=1,\quad a_0+ib_0=0.
\end{align*}
Combined with the fact that the right-hand side also vanishes identically, this leads to the following identity:
\begin{align*}
\left(c_0^2+d_0^2\right)\left(\wp(z)-\wp(q_3)\right)&+2\beta(c_0-id_0)\wp(z)\left(\wp(z)-\wp(q_3)\right)\\
&+2\gamma(c_0-i\sigma_3d_0)+2\beta\gamma(1-\sigma_3)\wp(z)\equiv0.
\end{align*}
Comparing coefficients in this identity, we obtain $c_0-id_0=0$ and $\sigma_3=1$. Therefore, by (\ref{vvv1}), we conclude that $\phi_1 \equiv -i \phi_2$ and $\phi_3 \equiv i \phi_4$, completing the proof.
\end{proof}

\begin{lemma}\label{etacase2}
Assume that
\begin{align*}
q_1\in I_H,\quad q_3 \in I_H.
\end{align*}
Then the Weierstrass data in (\ref{vvv1}) satisfy the identity $\sum_{j=1}^4\phi_j^2\equiv0$ if and only if one of the following holds:
\begin{itemize}
\item[(1)] $\phi_1 \equiv -i \phi_2$ and $\phi_3 \equiv i \phi_4$, in which case the corresponding minimal surface is $J$-holomorphic for some almost complex structure $J$.
\item[(2)] The coefficients $a_0,b_0,c_0,d_0$ are given by
\begin{align*}
a_0&=-\beta\wp(q_1)-\frac{\alpha + \gamma}{\wp(q_3)-\wp(q_1)},\\
b_0&= i\beta\wp(q_1)-i\frac{\alpha+\gamma}{\wp(q_3)-\wp(q_1)},\\
c_0&=-\beta\wp(q_3)+\frac{\alpha + \gamma}{\wp(q_3)-\wp(q_1)},\\
d_0&=-i\beta\wp(q_3)-i\frac{\alpha+\gamma}{\wp(q_3)-\wp(q_1)}.
\end{align*}
In this case, $\sigma_2=-1$ and $\sigma_3=1$, and hence the values of the generalized Gauss map at $q_2$ and $q_3$ are
\begin{align*}
\Phi(q_2)=[1,-i,1,i]\in\mathbb{P}^3,\quad \Phi(q_3)=[0,0,1,-i]\in\mathbb{P}^3.
\end{align*} 
\end{itemize}
\end{lemma}

\begin{proof}
In the present case, we have $\eta_1=H_1(z)dz$ and $\eta_3=H_3(z)dz$. We proceed similarly to the proof of Lemma \ref{etacase1}. Using the identity $\sum_{j=1}^4\phi_j^2\equiv0$, we compute
\begin{align*}
\sum_{j=1}^4\phi_j^2&=\left[\left(a_0^2+b_0^2+c_0^2+d_0^2\right)+2\alpha(a_0+ib_0)\frac{1}{\wp(z)-\wp(q_1)}\right.\\
&\qquad+2\beta(a_0+i\sigma_2b_0+c_0-i\sigma_2d_0)\wp(z)+2\gamma(c_0-i\sigma_3d_0)\frac{1}{\wp(z)-\wp(q_3)}\\
&\qquad\left.+2\alpha\beta(1-\sigma_2)\wp(z)\frac{1}{\wp(z)-\wp(q_1)}+2\beta\gamma(1-\sigma_2\sigma_3)\wp(z)\frac{1}{\wp(z)-\wp(q_3)}\right](dz)^2\\
&\equiv0.
\end{align*}
Multiplying the above identity by $\left(\wp(z)-\wp(q_1)\right)\left(\wp(z)-\wp(q_3)\right)$, we obtain
\begin{align*}
&\left(a_0^2+b_0^2+c_0^2+d_0^2\right)\left(\wp(z)-\wp(q_1)\right)\left(\wp(z)-\wp(q_3)\right)\\
&\qquad+2\alpha(a_0+ib_0)\left(\wp(z)-\wp(q_3)\right)\\
&\qquad+2\beta(a_0+i\sigma_2b_0+c_0-i\sigma_2d_0)\wp(z)\left(\wp(z)-\wp(q_1)\right)\left(\wp(z)-\wp(q_3)\right)\\
&\qquad+2\gamma(c_0-i\sigma_3d_0)\left(\wp(z)-\wp(q_1)\right)\\
&\qquad+2\alpha\beta(1-\sigma_2)\wp(z)\left(\wp(z)-\wp(q_3)\right)\\
&\qquad+2\beta\gamma(1-\sigma_2\sigma_3)\wp(z)\left(\wp(z)-\wp(q_1)\right)\equiv0.
\end{align*}
Since $\wp(z)$ takes infinitely many values, this identity must hold as a polynomial in $\wp(z)$. Therefore, the coefficient of the leading term $\wp(z)^3$ must vanish, which gives
\begin{align}\label{YYYYYY1}
a_0+i\sigma_2b_0+c_0-i\sigma_2d_0=0
\end{align}
as $\beta\neq 0$.

Substituting (\ref{YYYYYY1}) into the identity gives
\begin{align*}
&\left(a_0^2+b_0^2+c_0^2+d_0^2\right)\left(\wp(z)-\wp(q_1)\right)\left(\wp(z)-\wp(q_3)\right)\\
&\qquad+2\alpha(a_0+ib_0)\left(\wp(z)-\wp(q_3)\right)\\
&\qquad+2\gamma(c_0-i\sigma_3d_0)\left(\wp(z)-\wp(q_1)\right)\\
&\qquad+2\alpha\beta(1-\sigma_2)\wp(z)\left(\wp(z)-\wp(q_3)\right)\\
&\qquad+2\beta\gamma(1-\sigma_2\sigma_3)\wp(z)\left(\wp(z)-\wp(q_1)\right)\equiv0.
\end{align*}
Evaluating this identity at $z=q_1$ and $z=q_3$, and using $\wp(q_1)\neq\wp(q_3)$, we obtain
\begin{align}\label{YYYYYY2}
a_0+ib_0+\beta(1-\sigma_2)\wp(q_1)=0
\end{align}
and
\begin{align}\label{YYYYYY3}
c_0-i\sigma_3d_0+\beta(1-\sigma_2\sigma_3)\wp(q_3)=0,
\end{align}
where we have used $\alpha,\gamma\neq0$. Substituting (\ref{YYYYYY2}) and (\ref{YYYYYY3}) into the identity yields 
\begin{align*}
&\left[\left(a_0^2+b_0^2+c_0^2+d_0^2\right)+2\alpha\beta(1-\sigma_2)+2\beta\gamma(1-\sigma_2\sigma_3)\right]\\
&\qquad\times\left(\wp(z)-\wp(q_1)\right)\left(\wp(z)-\wp(q_3)\right)\equiv0,
\end{align*}
and therefore
\begin{align}\label{YYYYYY4}
a_0^2+b_0^2+c_0^2+d_0^2+2\alpha\beta(1-\sigma_2)+2\beta\gamma(1-\sigma_2\sigma_3)=0.
\end{align}
We now analyze equations (\ref{YYYYYY1}), (\ref{YYYYYY2}), (\ref{YYYYYY3}), and (\ref{YYYYYY4}) according to the possible values of $\sigma_2$ and $\sigma_3$.

If $\sigma_2=1$ and $\sigma_3=1$, then (\ref{YYYYYY2}) and (\ref{YYYYYY3}) give $a_0+ib_0=0$ and $c_0-id_0=0$. Hence, from (\ref{vvv1}), we see that $\phi_1 \equiv -i \phi_2$ and $\phi_3 \equiv i \phi_4$.

If $\sigma_2=1$ and $\sigma_3=-1$, then (\ref{YYYYYY1}) and (\ref{YYYYYY2}) imply that $a_0+ib_0=0$ and $c_0-id_0=0$. Substituting these into (\ref{YYYYYY4}) yields $4\beta\gamma=0$, contradicting $\beta, \gamma\neq0$. 

If $\sigma_2=-1$ and $\sigma_3=-1$, then (\ref{YYYYYY1}) and (\ref{YYYYYY3}) imply that $a_0-ib_0=0$ and $c_0+id_0=0$. Substituting these into (\ref{YYYYYY4}) gives $4\alpha\beta=0$, contradicting $\alpha, \beta\neq0$.

Finally, suppose that $\sigma_2=-1$ and $\sigma_3=1$. Then (\ref{YYYYYY2}) and (\ref{YYYYYY3}) give
\begin{align*}
a_0+ib_0=-2\beta\wp(q_1),\quad c_0-id_0=-2\beta\wp(q_3).
\end{align*}
Together with
\begin{align*}
a_0-ib_0=-c_0-id_0,
\end{align*}
which follows from (\ref{YYYYYY1}), equation (\ref{YYYYYY4}) gives
\begin{align*}
0&=a_0^2+b_0^2+c_0^2+d_0^2+4\alpha\beta+4\beta\gamma\\
&=(a_0+ib_0)(a_0-ib_0)+(c_0-id_0)(c_0+id_0)+4\beta(\alpha+\gamma)\\
&=2\beta\wp(q_1)(c_0+id_0)-2\beta\wp(q_3)(c_0+id_0)+4\beta(\alpha+\gamma).
\end{align*}
Thus,
\begin{align*}
c_0+id_0=\frac{2(\alpha+\gamma)}{\wp(q_3)-\wp(q_1)}.
\end{align*}
Combining these identities, we obtain the coefficients $a_0,b_0,c_0,d_0$ given in case $(2)$. Conversely, it is straightforward to verify that the Weierstrass data given by (\ref{vvv1}) with these coefficients satisfy the identity $\sum_{j=1}^4\phi_j^2\equiv0$, completing the proof.
\end{proof}

\subsection{Period computation}\label{periodcompute}
In Lemma \ref{etacase1} and Lemma \ref{etacase2}, we showed that the Weierstrass data (\ref{vvv1}) satisfying $\sum_{j=1}^4\phi_j^2\equiv0$ either correspond to a $J$-holomorphic curve or fall into case $(2)$ of Lemma \ref{etacase2}. In this subsection, we examine the vanishing real period conditions in case $(2)$ and derive equivalent conditions on $\tau$ and the choices of $q_1,q_3$ under which the data correspond to well-defined surfaces satisfying the assumptions of Theorem \ref{thm1}.

Substituting the coefficients in case $(2)$ of Lemma \ref{etacase2} into (\ref{vvv1}), the Weierstrass data are given by
\begin{equation}\label{periodcompdata}
\begin{alignedat}{4}
\phi_1
&= \left(-\beta\wp(q_1)-\frac{\alpha+\gamma}{\wp(q_3)-\wp(q_1)}\right)dz
&{}+{}& \alpha\eta_1
&{}+{}& \beta\eta_2
&{}+{}& 0,\\
\phi_2
&= \left(i\beta\wp(q_1)-i\frac{\alpha+\gamma}{\wp(q_3)-\wp(q_1)}\right)dz
&{}+{}& i\alpha\eta_1
&{}-{}& i\beta\eta_2
&{}+{}& 0,\\
\phi_3
&= \left(-\beta\wp(q_3)+\frac{\alpha+\gamma}{\wp(q_3)-\wp(q_1)}\right)dz
&{}+{}& 0
&{}+{}& \beta\eta_2
&{}+{}& \gamma\eta_3,\\
\phi_4
&= \left(-i\beta\wp(q_3)-i\frac{\alpha+\gamma}{\wp(q_3)-\wp(q_1)}\right)dz
&{}+{}& 0
&{}+{}& i\beta\eta_2
&{}-{}& i\gamma\eta_3,
\end{alignedat}
\end{equation}
for some nonzero complex numbers $\alpha, \beta, \gamma$. Since $q_1, q_3\in I_H$ and $q_2$ has been identified with $\overline{0}\in\mathbb{C}/\Lambda(1,\tau)$, we have
\begin{align*}
\eta_1 = H_1(z)dz=\frac{1}{\wp(z) - \wp(q_1)} dz,\quad \eta_3 =H_3(z)dz= \frac{1}{\wp(z) - \wp(q_3)}dz,
\end{align*} 
and
\begin{align*}
\eta_2=\wp(z)dz.
\end{align*}

Consider the paths $C_1, C_2:[0,1]\to \mathbb{C}$ given by 
\begin{align*}
C_1(t) = t+v_0 \tau , \quad C_2(t) = u_0 + t \tau, 
\end{align*}
where $0<u_0,v_0<1$. Their projections onto $\mathbb{C}/\Lambda(1,\tau)$ form a homology basis. Define
\begin{align*}
\mu:=\int_{C_1}\wp(z) dz &= \zeta(v_0\tau) - \zeta (1+v_0 \tau),
\end{align*} 
where $\zeta$ denotes the Weierstrass zeta function. The Legendre relation then gives
\begin{align*}
\int_{C_2}\wp(z) dz &= \zeta(u_0) - \zeta (u_0+ \tau) =\tau \mu + 2\pi i.
\end{align*}

Let $q_4$ be the remaining half-period in 
\begin{align*}
\left\{\frac{1}{2}, \frac{\tau}{2}, \frac{1+\tau}{2}\right\} \setminus \{q_1, q_3\},
\end{align*}
and denote
\begin{align*}
\wp_1:=\wp(q_1),\quad \wp_3:=\wp(q_3),\quad \wp_4:=\wp(q_4).
\end{align*}
We then define the matrix $M_{\tau,\wp_1,\wp_3}$ by
\begin{align*}
M_{\tau,\wp_1,\wp_3}=\begin{pmatrix}
\overline{ \wp_1 -\mu} &  -\frac{\mu - \wp_1}{(\wp_4 - \wp_1)(\wp_3 - \wp_1)} + \frac{1}{\wp_3 - \wp_1} &\frac{1}{\wp_3 - \wp_1}  \\
\overline{\wp_3 -  \mu} & -\frac{1 }{\wp_3 - \wp_1}  &-\frac{\mu-\wp_3}{(\wp_4 - \wp_3)(\wp_1 - \wp_3)} - \frac{1}{\wp_3 - \wp_1}  \\
\overline{ \wp_1 \tau -\tau\mu - 2\pi i} & -\frac{\tau(\mu-\wp_1) + 2\pi i }{(\wp_4 - \wp_1)(\wp_3 - \wp_1)} + \frac{\tau}{\wp_3 - \wp_1} & \frac{\tau}{\wp_3 -\wp_1}  \\
 \overline{ \wp_3 \tau- \tau\mu - 2\pi i}& -\frac{\tau}{\wp_3 - \wp_1}  & -\frac{\tau(\mu - \wp_3) + 2\pi i}{(\wp_4 - \wp_3)(\wp_1 - \wp_3)}- \frac{\tau}{\wp_3 - \wp_1}
\end{pmatrix}.
\end{align*}

\begin{lemma} \label{periodlemma}
The Weierstrass data (\ref{periodcompdata}) have vanishing real periods if and only if nonzero complex numbers $\alpha,\beta,\gamma$ in (\ref{periodcompdata}) satisfy 
\begin{align}\label{period condition}
M_{\tau,\wp_1,\wp_3}\cdot\begin{pmatrix}
\bar{\beta} \\
\alpha  \\
\gamma
\end{pmatrix} = 0.
\end{align}
\end{lemma}

\begin{proof}
First, we recall the following identities for the Weierstrass $\wp$-function. They follow directly from the addition formula for $\wp$ \cite[equation (4), p.~82]{S}:
\begin{align*}
\frac{1}{\wp(z) - e_1} &= \frac{1}{(e_3- e_1)(e_2 - e_1)}\bigg{[} \wp \left( z - \frac{1}{2} \right)- e_1\bigg{]},\\
\frac{1}{\wp(z) - e_2} &= \frac{1}{(e_3- e_2)(e_1 - e_2)}\bigg{[}\wp\left( z - \frac{\tau}{2}\right)-e_2\bigg{]},\\
\frac{1}{\wp(z) - e_3} &= \frac{1}{(e_1- e_3)(e_2 - e_3)}\bigg{[}\wp \left( z - \frac{1+\tau}{2}\right)-e_3\bigg{]}.
\end{align*}
Using these identities, we obtain
\begin{align*}
\int_{C_1} \frac{1}{\wp(z) - e_1} dz =  \frac{\mu - e_1}{(e_3- e_1)(e_2 - e_1)} , \quad \int_{C_2} \frac{1}{\wp(z) - e_1} dz =  \frac{\tau\mu + 2\pi i - \tau e_1}{(e_3- e_1)(e_2 - e_1)}, \\
\int_{C_1} \frac{1}{\wp(z) - e_2} dz =  \frac{\mu - e_2}{(e_1- e_2)(e_3 - e_2)} ,\quad\int_{C_2} \frac{1}{\wp(z) - e_2} dz =  \frac{\tau\mu + 2\pi i - \tau e_2}{(e_1- e_2)(e_3 - e_2)},\\
\int_{C_1} \frac{1}{\wp(z) - e_3} dz =  \frac{\mu - e_3}{(e_1- e_3)(e_2 - e_3)} ,\quad\int_{C_2} \frac{1}{\wp(z) - e_3} dz =  \frac{\tau\mu + 2\pi i - \tau e_3}{(e_1- e_3)(e_2 - e_3)}.
\end{align*}

We can therefore compute the periods of the $1$-forms $\phi_j$ along $C_1$ and $C_2$ as follows:
\begin{align*}
\int_{C_1}\phi_1 &= -\frac{\alpha + \gamma}{\wp_3 - \wp_1 }- \beta \wp_1 + \alpha \frac{\mu - \wp_1}{(\wp_4 -\wp_1)(\wp_3 - \wp_1)} +\beta \mu,\\
\int_{C_2}\phi_1 &= -\frac{(\alpha + \gamma)\tau}{\wp_3 - \wp_1 }- \beta \wp_1\tau + \alpha \frac{\tau\mu +2\pi i - \tau \wp_1}{(\wp_4 -\wp_1)(\wp_3 - \wp_1)} +\beta (\tau\mu + 2\pi i),\\
\int_{C_1}\phi_2 &= -i \frac{\alpha + \gamma}{\wp_3 - \wp_1 }+i \beta \wp_1 +i \alpha \frac{\mu - \wp_1}{(\wp_4 -\wp_1)(\wp_3 - \wp_1)}-i\beta \mu,\\
\int_{C_2}\phi_2 &= -i\frac{(\alpha + \gamma)\tau}{\wp_3 - \wp_1 }+i \beta \wp_1\tau + i\alpha \frac{\tau\mu +2\pi i - \tau \wp_1}{(\wp_4 -\wp_1)(\wp_3 - \wp_1)} -i\beta (\tau\mu + 2\pi i),\\
\int_{C_1}\phi_3 &= \frac{\alpha + \gamma}{\wp_3 - \wp_1 }- \beta \wp_3+\beta\mu +\gamma \frac{\mu - \wp_3}{(\wp_4 - \wp_3)(\wp_1 - \wp_3)},\\
\int_{C_2}\phi_3 &= \frac{(\alpha + \gamma)\tau}{\wp_3 - \wp_1 }- \beta \wp_3 \tau +\beta (\tau\mu + 2\pi i) +\gamma \frac{\tau \mu + 2\pi i -\tau \wp_3}{(\wp_4 - \wp_3)(\wp_1 - \wp_3)} ,\\
\int_{C_1}\phi_4 &= -i \frac{\alpha + \gamma}{\wp_3 - \wp_1 }-i \beta \wp_3 +i \beta \mu-i\gamma \frac{\mu-\wp_3}{(\wp_4 - \wp_3)(\wp_1 - \wp_3)} ,\\
\int_{C_2}\phi_4 &=- i\frac{(\alpha + \gamma)\tau}{\wp_3 - \wp_1 }  -i\beta \wp_3\tau+i\beta(\tau\mu + 2\pi i) -i\gamma \frac{\tau \mu + 2\pi i -\tau \wp_3}{(\wp_4 - \wp_3)(\wp_1 - \wp_3)} .
\end{align*}

Combining the vanishing real parts of the $C_1$-periods of $\phi_1$ and $\phi_2$, we obtain
\begin{align*}
\overline{\beta \wp_1 - \beta \mu} - \alpha\left( \frac{\mu - \wp_1}{(\wp_4 - \wp_1)(\wp_3 - \wp_1)} - \frac{1}{\wp_3 - \wp_1}\right) + \gamma \left( \frac{1}{\wp_3 - \wp_1}\right)=0.
\end{align*}
Similarly, the $C_2$-periods of $\phi_1,\phi_2$ and the $C_1$- and $C_2$-periods of $\phi_3,\phi_4$ yield the remaining three conditions:
\begin{align*}
&\overline{\beta \wp_1 \tau - \beta (\tau\mu + 2\pi i)} - \alpha\left( \frac{\tau(\mu-\wp_1) + 2\pi i }{(\wp_4 - \wp_1)(\wp_3 - \wp_1)} - \frac{\tau}{\wp_3 - \wp_1}\right) + \gamma \left( \frac{\tau}{\wp_3 -\wp_1}\right)=0,\\
&\overline{\beta \wp_3 - \beta \mu}  - \gamma \left( \frac{\mu - \wp_3}{(\wp_4 - \wp_3)(\wp_1 - \wp_3)}\right)- \frac{\alpha + \gamma}{\wp_3 - \wp_1}=0,\\
&\overline{\beta \wp_3 \tau- \beta (\tau\mu + 2\pi i)}  - \gamma \left( \frac{\tau(\mu - \wp_3) + 2\pi i}{(\wp_4 - \wp_3)(\wp_1 - \wp_3)}\right)- \frac{(\alpha + \gamma)\tau}{\wp_3 - \wp_1}=0.
\end{align*}

Rearranging the terms, we derive the system
\begin{align*}
&\overline{\beta (\wp_1 -\mu)} - \alpha\left( \frac{\mu - \wp_1}{(\wp_4 - \wp_1)(\wp_3 - \wp_1)} - \frac{1}{\wp_3 - \wp_1}\right) + \gamma \left( \frac{1}{\wp_3 - \wp_1}\right)=0,\\
&\overline{\beta (\wp_3 -  \mu)} -\alpha \frac{1 }{\wp_3 - \wp_1} - \gamma \left( \frac{\mu - \wp_3}{(\wp_4 - \wp_3)(\wp_1 - \wp_3)} + \frac{1}{\wp_3 - \wp_1} \right)=0,\\
&\overline{\beta (\wp_1 \tau -\tau\mu - 2\pi i)} - \alpha\left( \frac{\tau(\mu-\wp_1) + 2\pi i }{(\wp_4 - \wp_1)(\wp_3 - \wp_1)} - \frac{\tau}{\wp_3 - \wp_1}\right) + \gamma \left( \frac{\tau}{\wp_3 -\wp_1}\right)=0,\\
&\overline{\beta (\wp_3 \tau- \tau\mu - 2\pi i)} - \alpha \frac{\tau}{\wp_3 - \wp_1} - \gamma \left( \frac{\tau(\mu - \wp_3) + 2\pi i}{(\wp_4 - \wp_3)(\wp_1 - \wp_3)}+ \frac{\tau}{\wp_3 - \wp_1}\right)=0.
\end{align*}
This system is precisely equivalent to 
\begin{align*}
M_{\tau,\wp_1,\wp_3}\cdot\begin{pmatrix}
\bar{\beta} \\
\alpha  \\
\gamma
\end{pmatrix} = 0.
\end{align*}
Conversely, reversing the above computations shows that this system implies the vanishing of all real periods. Hence the proof is complete.
\end{proof}

A necessary condition for (\ref{period condition}) to admit a nontrivial solution is 
\begin{align*}
\textup{rk}M_{\tau,\wp_1,\wp_3} \leq 2.
\end{align*}
In the next lemma, we show that 
\begin{align*}
\textup{rk}M_{\tau,\wp_1,\wp_3} \geq 2
\end{align*}
and derive conditions equivalent to $\textup{rk}M_{\tau,\wp_1,\wp_3}=2$.

\begin{lemma}\label{rank}
We have
\begin{align*}
\textup{rk}M_{\tau,\wp_1,\wp_3}\geq 2.
\end{align*}
Moreover, equality holds if and only if the following two conditions are satisfied:
\begin{enumerate}
\item $|\wp_1| = |\wp_3|,$ 
\item $\wp_1 -\mu+ \overline{\wp_3-\mu} = \frac{\textup{Im} \tau}{\pi}\left( |\wp_1-\mu|^2 +(\wp_4-\mu) \overline{\wp_3 - \wp_1}\right).$
\end{enumerate}
\end{lemma}

\begin{proof}
For convenience, define 
\begin{align*}
A:=\wp_1-\mu, \quad B:= \wp_3- \wp_1, \quad C := \wp_4 - \wp_1.
\end{align*}
Applying elementary row operations, we have
\begin{align}\label{Mrow}
\textup{rk}M_{\tau,\wp_1,\wp_3} &=\textup{rk}\begin{pmatrix}
\overline{ \wp_1 -\mu} & \frac{ \wp_4-\mu}{(\wp_3-\wp_1)(\wp_4-\wp_1)}&\frac{1}{\wp_3-\wp_1}  \\
\overline{\wp_3 - \mu} & -\frac{1}{\wp_3-\wp_1} & \frac{\wp_4 - \mu}{(\wp_4-\wp_3)(\wp_1-\wp_3)} \\
\frac{\textup{Im} \tau}{\pi} \overline{\wp_1 - \mu}-1 & \frac{1}{(\wp_3-\wp_1)(\wp_4-\wp_1)} & 0 \\
\frac{\textup{Im} \tau}{\pi} \overline{\wp_3 - \mu}-1& 0  & \frac{1}{(\wp_4-\wp_3)(\wp_1-\wp_3)}
\end{pmatrix}\nonumber\\
&= \textup{rk}\begin{pmatrix}
\overline{A} &  \frac{A+C}{BC}&\frac{1}{B}  \\
\overline{A}+\overline{B} & -\frac{1}{B} &\frac{A+C}{B(B-C)} \\
\frac{\textup{Im} \tau}{\pi} \overline{A}-1 & \frac{1}{BC} & 0 \\
\frac{\textup{Im} \tau}{\pi} (\overline{A}+\overline{B})-1& 0  & \frac{1}{B(B-C)}
\end{pmatrix}. 
\end{align}
Further elementary column operations give
\begin{align*}
\textup{rk}M_{\tau,\wp_1,\wp_3}= \textup{rk}\begin{pmatrix}
\overline{A} &  A+C&B-C  \\
\overline{A}+\overline{B} & -C &A+C \\
\frac{\textup{Im} \tau}{\pi} \overline{A}-1 & 1 & 0 \\
\frac{\textup{Im} \tau}{\pi} (\overline{A}+\overline{B})-1& 0  & 1
\end{pmatrix}. 
\end{align*}

The last two rows are linearly independent, and hence
\begin{align*}
\textup{rk}M_{\tau,\wp_1,\wp_3} \geq 2.
\end{align*}
Therefore, $\textup{rk}M_{\tau,\wp_1,\wp_3} = 2$ if and only if the first two rows are linear combinations of the last two rows. Equivalently,
\begin{align*}
\textup{rk}M_{\tau,\wp_1,\wp_3} =2 \Leftrightarrow   \left\{\begin{matrix}
\bar{A} = \left( \frac{\text{Im}\tau}{\pi} \bar{A} -1\right) (A+C) + \left( \frac{\text{Im}\tau}{\pi} (\bar{A}+\bar{B}) -1\right)(B-C),\\ \bar{A}+\bar{B} = - \left(\frac{\text{Im}\tau}{\pi} \bar{A} -1 \right)C + \left(\frac{\text{Im}\tau}{\pi} (\bar{A}+\bar{B}) -1\right)(A+C).
\end{matrix}\right.
\end{align*}

Expanding these two equations yields
\begin{align}\label{expand1}
A+\bar{A} +B = \frac{\text{Im}\tau}{\pi} \left( |A|^2 + |B|^2 +\bar{A}B -C\bar{B}\right)
\end{align}
and
\begin{align}\label{expand2}
A+\bar{A} +\bar{B} = \frac{\text{Im}\tau}{\pi} \left( |A|^2 + A\bar{B} +C\bar{B}\right). 
\end{align}
Taking the complex conjugate of (\ref{expand1}) and subtracting (\ref{expand2}), we obtain
\begin{align}\label{expand3}
|B|^2 = C\bar{B}+\bar{C}B.
\end{align}

Equation (\ref{expand2}) directly gives 
$$\wp_1 -\mu+ 
\overline{\wp_3-\mu} = \frac{\textup{Im} \tau}{\pi}\left( |\wp_1-\mu|^2 +(\wp_4-\mu) \overline{\wp_3 - \wp_1}\right).$$
On the other hand, (\ref{expand3}) implies 
\begin{align*}
|\wp_3 -\wp_1|^2 =(\wp_4 - \wp_1)(\overline{\wp_3}-\overline{\wp_1})+(\overline{\wp_4}-\overline{\wp_1})(\wp_3 - \wp_1).
\end{align*}
Using the identity $\wp_1+\wp_3+\wp_4 =0$, the right-hand side becomes
\begin{align*}
-|\wp_3|^2 - 2 \wp_1\overline{\wp_3} + \wp_3\overline{\wp_1}+ 2|\wp_1|^2 - |\wp_3|^2 -2\overline{\wp_1}\wp_3 + \overline{\wp_3}\wp_1 + 2|\wp_1|^2.
\end{align*}
Comparing this with
\begin{align*}
|\wp_3 -\wp_1|^2=|\wp_3|^2 + |\wp_1|^2 - \wp_3\overline{\wp_1} - \overline{\wp_3}\wp_1,
\end{align*}
we obtain 
$$ |\wp_1|^2 = |\wp_3|^2,$$
and hence $|\wp_1| = |\wp_3|$.

Consequently,
$$\textup{rk}M_{\tau,\wp_1,\wp_3} = 2$$ if and only if
\begin{align*}
\begin{cases}
|\wp_1| = |\wp_3|, \\
\wp_1-\mu + 
\overline{\wp_3-\mu}= \frac{\textup{Im} \tau}{\pi}\left( |\wp_1-\mu|^2 +(\wp_4-\mu) \overline{\wp_3 - \wp_1}\right).
\end{cases}
\end{align*}
\end{proof}

\begin{lemma}\label{lemmakernel}
Suppose that $\textup{rk}M_{\tau,\wp_1,\wp_3} = 2$, and consider a nontrivial solution 
$$\begin{pmatrix}
\bar{\beta} \\
\alpha  \\
\gamma
\end{pmatrix} \in \textup{ker}M_{\tau,\wp_1,\wp_3}$$
for some complex numbers $\alpha,\beta,\gamma$. Then $\alpha \beta \gamma \neq 0$.
\end{lemma}

\begin{proof}
Suppose, to the contrary, that $\alpha\beta\gamma=0$. We use the same notation as in the proof of the previous lemma. Since the kernel of $M_{\tau, \wp_1, \wp_3}$ is invariant under elementary row operations, (\ref{Mrow}) implies
\begin{align*}
&\left(\frac{\text{Im}\tau }{\pi} \overline{A} -1\right)\overline{\beta} + \frac{1}{BC}\alpha = 0,\\
&\left( \frac{\text{Im}\tau}{\pi} \left( \overline{A}+\overline{B} \right) -1 \right) \overline{\beta} + \frac{1}{B(B-C)}\gamma =0.
\end{align*}
If $\overline{\beta}=0$, then $\alpha =\gamma =0$, contradicting the assumption that the solution is nontrivial. Hence $\overline{\beta}\neq 0$, and it follows from the above two equations that 
$$\alpha \beta \gamma =B^2C(B-C) \left(\frac{\text{Im}\tau }{\pi} \overline{A} -1\right) \left( \frac{\text{Im}\tau}{\pi} \left( \overline{A}+\overline{B} \right) -1 \right) |\beta|^2 \overline{\beta}=0. $$
Since $B$, $C$, and $B-C$ are nonzero, at least one of the following relations must hold:
$$ \overline{A} = \frac{\pi}{\text{Im}\tau}, \quad  \overline{A}+\overline{B} = \frac{\pi}{\text{Im}\tau}. $$ 
Using (\ref{expand2}), either relation yields
\begin{align}\label{expand4}
C\overline{B} = \left (\frac{\pi}{\text{Im}\tau}\right)^2.
\end{align}

In particular, $C\overline{B}\in\mathbb{R}$. Combining this with (\ref{expand3}), we have
\begin{align*}
|B|^2=2C\overline{B}.
\end{align*} 
Since $B\neq 0$, it follows that $B=2C$. Recalling that $B=\wp_3-\wp_1$ and $C=\wp_4-\wp_1$, we find
\begin{align*}
\wp_1 + \wp_3 = 2\wp_4.
\end{align*}
Together with the identity $\wp_1+\wp_3+\wp_4=0$, this gives
\begin{align}\label{final relation} 
\wp_4 = 0, \quad \wp_1+\wp_3 = 0.
\end{align}
Substituting (\ref{final relation}) into (\ref{expand4}), we obtain
\begin{align}\label{relation4}
\left( \frac{\pi}{\textup{Im}\tau}\right)^2 = 2|\wp_1|^2.
\end{align}

It is known that $\wp_4=0$ implies that $\tau$ is $\text{SL}_2(\mathbb{Z})$-equivalent to $i$ (see \cite{EZ}). Thus, write
\begin{align*}
\tau = \frac{ai+b}{ci+d}
\end{align*}
for some $\begin{pmatrix} 
a&b\\
c&d
\end{pmatrix} \in \text{SL}_2(\mathbb{Z})$. By the transformation formula for the Weierstrass $\wp$-function,
\begin{align*}
\wp\left(q_1;\tau\right)=\wp\left(q_1 ; \frac{ai+b}{ci+d}\right) = (ci+d)^2\wp((ci+d)q_1;i).
\end{align*}
Moreover, since $ad-bc=1$, $$\text{Im} \tau=\text{Im}\left(\frac{ai+b}{ci+d}\right)= \frac{1}{|ci+d|^2}.$$ Consequently,
\begin{align}\label{expand9}
\left(\text{Im}\tau\right)^2\left|\wp(q_1;\tau)\right|^2=\left|\wp((ci+d)q_1;i)\right|^2.
\end{align}
Since $q_1\in\left\{\frac{1}{2}, \frac{\tau}{2}, \frac{1+\tau}{2}\right\}$, it follows that
\begin{align*}
(ci+d)q_1\in\left\{\frac{ci+d}{2}, \frac{ai+b}{2}, \frac{(a+c)i+(b+d)}{2}\right\}.
\end{align*}
Hence (\ref{relation4}) and (\ref{expand9}) would imply that one of the three half-period values for the square lattice has squared modulus $\frac{\pi^2}{2}$.

However, for the square lattice $\Lambda(1,i)$, 
\begin{align*}
\left|\wp\left(\frac{1+i}{2};i\right)\right|^2=0,
\end{align*}
whereas
\begin{align*}
\left|\wp\left(\frac{1}{2};i\right)\right|^2=\left|\wp\left(\frac{i}{2};i\right)\right|^2=\left( \frac{\Gamma(\frac{1}{4})^4}{8\pi} \right)^2\neq\frac{\pi^2}{2}.
\end{align*}
Thus (\ref{relation4}) is impossible. This contradiction proves that $\alpha\beta\gamma \neq 0$.
\end{proof}

The above lemmas together yield the following result.
\begin{lemma}\label{2gutdolemmadaimma}
There exists a minimal surface satisfying the assumptions of Theorem \ref{thm1} with Weierstrass data of the form (\ref{periodcompdata}) if and only if
\begin{align*}
\textup{rk}M_{\tau,\wp_1,\wp_3}=2.
\end{align*}
\end{lemma}
\begin{proof}
Suppose $\textup{rk}M_{\tau,\wp_1,\wp_3}=2$. Then $\textup{ker}M_{\tau,\wp_1,\wp_3}$ is nontrivial. Consider a nontrivial element
\begin{align*}
\begin{pmatrix}
\bar{\beta} \\
\alpha  \\
\gamma
\end{pmatrix} \in \textup{ker}M_{\tau,\wp_1,\wp_3}.
\end{align*}
By Lemma \ref{lemmakernel}, $\alpha,\beta,\gamma$ are all nonzero. Lemma \ref{periodlemma} then implies that the Weierstrass data (\ref{periodcompdata}) with these $\alpha,\beta,\gamma$ have vanishing real periods and hence define a well-defined minimal surface satisfying the assumptions of Theorem \ref{thm1}.

Conversely, suppose that there exists a minimal surface satisfying the assumptions of Theorem \ref{thm1} with Weierstrass data of the form (\ref{periodcompdata}) for some nonzero complex numbers $\alpha,\beta,\gamma$. By Lemma \ref{periodlemma}, 
\begin{align*}
\begin{pmatrix}
\bar{\beta} \\
\alpha  \\
\gamma
\end{pmatrix} \in \textup{ker}M_{\tau,\wp_1,\wp_3}.
\end{align*}
Since this vector is nontrivial, we have
\begin{align*}
\textup{rk}M_{\tau,\wp_1,\wp_3}\leq 2.
\end{align*}
On the other hand, Lemma \ref{rank} gives 
\begin{align*}
\textup{rk}M_{\tau,\wp_1,\wp_3}\geq 2.
\end{align*}
Therefore, $\textup{rk}M_{\tau,\wp_1,\wp_3}=2$.
\end{proof}

\subsection{Proof of Theorem \ref{thm1}}
Using properties of elliptic functions, we prove the following lemma.
\begin{lemma}\label{lemmadaimma}
For every $\tau\in F$ and every choice of half-period values $\wp_1, \wp_3$ of the Weierstrass $\wp$-function $\wp(z;\tau)$, 
\begin{align*}
\textup{rk}M_{\tau,\wp_1,\wp_3}\neq 2.
\end{align*}
\end{lemma}
Assuming this lemma, we can now complete the proof of Theorem \ref{thm1}.

\begin{proof}[Proof of Theorem \ref{thm1}]
Suppose that $S_1$ is represented via the generalized Weierstrass representation (\ref{gwr}). Then the associated data can be written as in (\ref{vvv1}). Examining the identity $\sum_{j=1}^4\phi_j^2\equiv0$, Lemma \ref{etacase1} and Lemma \ref{etacase2} show that the data either correspond to a $J$-holomorphic curve or fall into case $(2)$ of Lemma \ref{etacase2}. The latter case is ruled out by Lemma \ref{2gutdolemmadaimma} and Lemma \ref{lemmadaimma}. Therefore, all possible Weierstrass data correspond to $J$-holomorphic curves, and hence $S_1$ is $J$-holomorphic for some almost complex structure $J$.
\end{proof}

We first prove the following result concerning the $j$-invariant.
\begin{lemma} \label{jreal}
Let $\wp(z;\tau)$ be the Weierstrass $\wp$-function satisfying
$$ \left(\wp'\right)^2 = 4(\wp-e_1)(\wp-e_2)(\wp-e_3)= 4\wp^3 - g_2 \wp -g_3 .$$ 
Assume that $|e_i |= |e_j| $ for some distinct $i,j \in \{1,2,3\}$. Then the $j$-invariant
$$ j(\tau)= 1728 \frac{g_2^3}{g_2^3 - 27 g_3^2}$$ 
is real-valued.
\end{lemma}

\begin{proof}
The numbers $e_1, e_2,e_3$ are the distinct roots of 
$$t^3- \frac{g_2}{4}t-\frac{g_3}{4}=0.$$ 
If $g_2g_3=0$, then $j(\tau)$ is either $0$ or $1728$, and the conclusion follows immediately. We may therefore assume that $g_2g_3\neq 0$.

By Cardano's formula, 
$$ \{ e_1, e_2, e_3\} = \{R_1+R_2, \omega^2 R_1+\omega R_2, \omega R_1+\omega^2 R_2\},$$
where 
\begin{align*} 
R_1 = \frac{1}{2} \sqrt[3]{g_3 + \sqrt{g_3^2-\frac{g_2^3}{27}}}, \quad R_2 = \frac{1}{2} \sqrt[3]{g_3 - \sqrt{g_3^2-\frac{g_2^3}{27}}},
\end{align*}
and
\begin{align*}
\omega = -\frac{1}{2} + \frac{\sqrt3}{2}i.
\end{align*}
Thus, if $|e_i |= |e_j| $ for some distinct $i,j \in \{1,2,3\}$, one of the following equalities holds:
\begin{align*}
|\omega^2 R_1+\omega R_2|&=|R_1 + R_2|, \\
|R_1 + R_2|&=|\omega R_1+\omega^2 R_2|, \\
|\omega R_1+\omega^2 R_2|&=|\omega^2 R_1+\omega R_2|. 
\end{align*}

A direct calculation shows that these equalities are equivalent, respectively, to
\begin{align*}
\overline{R_1}R_2 &= \omega R_1 \overline{R_2}, \\
\overline{R_1}R_2 &= \omega^2 R_1 \overline{R_2}, \\
\overline{R_1}R_2 &= R_1 \overline{R_2}.
\end{align*}
Taking the cube in each case gives 
$$\left( \frac{\ R_1\ }{\overline{R_1}}\right)^3 =\left( \frac{\ R_2\ }{\overline{R_2}}\right)^3. $$
Hence $\left(\frac{R_1}{R_2}\right)^3$ is real. This implies that 
$$\arg\left( g_3 + \sqrt{g_3^2-\frac{g_2^3}{27}}\right) - \arg\left( g_3 - \sqrt{g_3^2-\frac{g_2^3}{27}}\right) = m\pi$$
for some integer $m$. Then, we see that $g_3^2$ and $g_2^3$ lie on the same line through the origin in the complex plane. It now follows from the definition that the $j$-invariant $j(\tau)$ is a real number.
\end{proof}
\begin{rmk}\label{remarkdaimma}\normalfont
The locus on which $j(\tau)$ takes real values is well known. For $\tau\in F$, 
\begin{align*}
j(\tau)\in \mathbb{R}
\end{align*}
if and only if 
\begin{align*}
\tau \in \partial F\quad\text{or}\quad \tau \in F\cap \{ z \in \mathbb{C} \ | \  \textup{Re}(z) = 0 \}.
\end{align*}
See \cite[Chapter 6.3.1--6.3.2]{D} for further details. 
\end{rmk}

The proof of Lemma \ref{lemmadaimma} will follow from the next three lemmas.

\begin{lemma}\label{lemAA}
Suppose that 
\begin{align*}
\tau=\frac{1}{2}+ci,\quad c\geq \frac{\sqrt{3}}{2}.
\end{align*}
If two half-period values $\wp_1,\wp_3$ of the Weierstrass $\wp$-function satisfy
\begin{align}\label{ponepthree}
|\wp_1| = |\wp_3|,
\end{align}
then
\begin{align}\label{donghwi}
\wp_1 -\mu+ 
\overline{\wp_3-\mu} = \frac{\textup{Im} \tau}{\pi}\left( | \wp_1-\mu|^2 +(\wp_4-\mu) \overline{\wp_3 - \wp_1}\right)
\end{align}
cannot hold, where $\wp_4$ is the remaining half-period value.
\end{lemma}
\begin{proof}
Using the Jacobi theta functions, the half-period values can be expressed as
\begin{align}  \label{thetae1}
&\left\{ 
\begin{aligned}
3e_1 &= \pi^2 \left( \vartheta_3^4 + \vartheta_4^4 \right), \\
3e_2 &= -\pi^2 \left( \vartheta_2^4 + \vartheta_3^4 \right), \\
3e_3 &= \pi^2 \left( \vartheta_2^4 - \vartheta_4^4 \right),
\end{aligned}
\right.
\end{align}
where  
\begin{align*} 
\vartheta_2 = 2 q^{\frac{1}{4}}\sum_{n=0}^\infty q^{n(n+1)},\quad \vartheta_3 = 1 + 2 \sum_{n=1}^\infty q^{n^2},\quad \vartheta_4 = 1 + 2 \sum_{n=1}^\infty (-1)^n q^{n^2},
\end{align*}
and 
\begin{align*}
q=e^{i\pi\tau}=ie^{-c\pi}.
\end{align*}
See \cite[p. 609]{C} or \cite[p. 294]{C2} for more details. In particular, $e_1>0$, while $e_2$ and $e_3$ are complex conjugates. 

First suppose that 
\begin{align*}
\left\{\wp_1,\wp_3\right\}=\{e_1, e_2\}\quad\text{or}\quad\left\{\wp_1,\wp_3\right\}=\{e_1, e_3\}.
\end{align*}
Since $e_2$ and $e_3$ are complex conjugates, (\ref{ponepthree}) implies that all three half-period values have the same modulus. Since $e_1+e_2+e_3=0$, it follows that
\begin{align*}
\left\{e_2, e_3\right\} = \left\{ e^{i\frac{2\pi}{3}}e_1,  e^{i\frac{4\pi}{3}}e_1 \right\}.
\end{align*}
Thus $g_2(\tau)=0$, and hence $j(\tau)=0$. Therefore,
\begin{align*}
\tau=\frac{1}{2}+\frac{\sqrt{3}}{2}i.
\end{align*}

At $\tau=\frac{1}{2}+\frac{\sqrt{3}}{2}i$, suppose first that 
\begin{align*}
\wp_1=e_1,\quad \wp_3=e^{\pm i\frac{2\pi}{3}}e_1. 
\end{align*}
Then 
\begin{align*}
\wp_4=e^{\mp i \frac{2\pi}{3}}e_1, 
\end{align*}
and (\ref{donghwi}) becomes
\begin{align} \label{donggri}
e_1-\mu + \overline{e^{\pm i\frac{2\pi}{3}}e_1-\mu} = \frac{\sqrt3}{2\pi} \left( |e_1-\mu|^2 +e_1\left(e^{\mp i\frac{2\pi}{3}}e_1-\mu \right) \overline{\left(e^{\pm i\frac{2\pi}{3}}-1 \right)}\right).
\end{align}
It is known that for $\tau=\frac{1}{2}+\frac{\sqrt3}{2} i$, the Eisenstein series of weight 2, $E_2(\tau)$, satisfies $$\mu = -\frac{\pi^2}{3}E_2(\tau) =-\frac{2\sqrt3\pi}{3}.$$ For further details on the Eisenstein series of weight $2$, refer to \cite{CL} and the references therein. Using this, comparing imaginary parts in (\ref{donggri}) gives
$$ \mp \frac{\sqrt3}{2}e_1 = \frac{\sqrt3}{2\pi}e_1 \left( \pm\frac{\sqrt3}{2} \mu \pm \sqrt{3}e_1\right)= \frac{\sqrt3}{2\pi}e_1 \left( \mp\pi \pm \sqrt3 e_1\right).$$
This simplifies to $e_1=0$, a contradiction. The same computation rules out the case
\begin{align*}
\wp_3=e_1,\quad \wp_1=e^{\pm i\frac{2\pi}{3}}e_1.
\end{align*}

It remains to consider
\begin{align*}
\left\{\wp_1,\wp_3\right\}=\{e_2, e_3\}.
\end{align*}
Suppose first that $\wp_1=e_2$ and $\wp_3=e_3$. Then $\wp_4=e_1$, and (\ref{donghwi}) becomes
\begin{align} \label{case1c}
e_2-\mu + \overline{ e_3-\mu} = \frac{c}{\pi} \left( |e_2-\mu|^2 + \left( e_1-\mu \right) \left( \overline{e_3} - \overline{e_2} \right) \right).
\end{align}
We recall the following identities (See \cite{C}, p. 608-610):
\begin{align} \label{qseries}
&\left\{ 
\begin{aligned}
-\mu + e_1 &= \pi^2 - 8\pi^2 \sum_{n=1}^\infty (-1)^n \frac{nq^{2n}}{1 - q^{2n}}, \\
-\mu + e_2 &= -8\pi^2 \sum_{n=1}^\infty \frac{nq^n}{1 - q^{2n}}, \\
-\mu + e_3 &= 8\pi^2 \sum_{n=1}^\infty \frac{(-1)^{n+1} n q^n}{1 - q^{2n}}.
\end{aligned}
\right.
\end{align}
Since $q$ is purely imaginary, substituting (\ref{qseries}) into (\ref{case1c}) and comparing imaginary parts yields
\begin{align*}
16\pi^2 \sum_{n=1}^\infty \frac{(2n-1) q^{2n-1}}{1 - q^{4n-2}} 
\cdot \left[ 
c \pi \left( 1 - 8 \sum_{n=1}^\infty (-1)^n \frac{n q^{2n}}{1 - q^{2n}} \right) - 1 
\right] = 0.
\end{align*}
Hence either
\begin{align}\label{e2e3}
\sum_{n=1}^\infty \frac{(2n-1) q^{2n-1}}{1 - q^{4n-2}} = 0,
\end{align}
or
\begin{align}\label{other}
c \left( 1 - 8 \sum_{n=1}^\infty (-1)^n \frac{n q^{2n}}{1 - q^{2n}} \right) = \frac{1}{\pi}.
\end{align}
We note that (\ref{e2e3}) cannot hold. Indeed, it would imply $e_2=e_3$ by (\ref{qseries}), which contradicts the fact that the values $e_1,e_2,e_3$ are distinct.

On the other hand, observe that (\ref{other}) is equivalent to 
\begin{align} \label{otherr}
c (-\mu + e_1) =\pi,
\end{align}
by the expression for \( -\mu + e_1 \) in \eqref{qseries}. However, \cite[Proposition 3(a)]{C} gives
\begin{align*}
c\text{Re}(-\mu+e_1)>2\pi.
\end{align*}
Since $-\mu+e_1\in\mathbb{R}$, this implies $c(-\mu+e_1)>2\pi$, which contradicts \eqref{otherr}. Thus this case is impossible, and the other case is also ruled out by the same computation. Therefore, the lemma is proved.
\end{proof}

\begin{lemma}\label{lemBB}
Suppose that
\begin{align*}
|\tau|=1,\quad \left|\textup{Re}\tau\right|<\frac{1}{2}.
\end{align*}
If two half-period values $\wp_1,\wp_3$ satisfy (\ref{ponepthree}), then (\ref{donghwi}) cannot hold.
\end{lemma}
\begin{proof}
We use the modular transformation
\begin{align*}
\begin{pmatrix}
0 & -1\\
1 & 0
\end{pmatrix}\cdot\tau=-\frac{1}{\tau}=-\overline{\tau},
\end{align*}
where the last equality follows from $|\tau|=1$. With the convention
\begin{align*}
e_1(\tau)=\wp\left(\frac{1}{2}; \tau\right),\quad e_2(\tau)=\wp\left(\frac{\tau}{2}; \tau\right),\quad e_3(\tau)=\wp\left(\frac{1+\tau}{2}; \tau\right),
\end{align*}
the transformation formula
\begin{align*}
\wp\left(\frac{z}{\tau};-\frac{1}{\tau}\right)=\tau^2\wp(z;\tau)
\end{align*}
gives
\begin{align*}
e_1\left(-\frac{1}{\tau}\right)=\tau^2e_2(\tau).
\end{align*}
Since $\Lambda(1,\overline{\tau})=\Lambda(1,-\overline{\tau})$, we have $\overline{e_1(\tau)}=e_1\left(\overline{\tau}\right)=e_1(-\overline{\tau})=e_1\left(-\frac{1}{\tau}\right)$. Therefore, 
\begin{align*}
\overline{e_1(\tau)}=\tau^2e_2(\tau).
\end{align*}
Similarly, 
\begin{align*}
\overline{e_2(\tau)}=\tau^2e_1(\tau),\quad \overline{e_3(\tau)}=\tau^2e_3(\tau).
\end{align*}
Moreover, since $\overline{E_2(\tau)}=E_2(-\overline{\tau})$, the quasimodular transformation law for $E_2$ yields
\begin{align*}
\overline{\mu(\tau)}=\tau^2\mu(\tau)+2\pi i\tau.
\end{align*}

In particular, $|e_1|=|e_2|$. If either $|e_1|=|e_3|$ or $|e_2|=|e_3|$, then all three half-period values have the same modulus. Together with $e_1+e_2+e_3=0$, this implies $g_2(\tau)=0$, and hence $j(\tau)=0$. Since $\tau$ lies in the open circular arc of the fundamental domain $F$, this is impossible. Therefore,
\begin{align*}
\left\{\wp_1,\wp_3\right\}=\{e_1,e_2\},\quad \wp_4=e_3.
\end{align*}

Assume first that $\wp_1=e_1$ and $\wp_3=e_2$. If (\ref{donghwi}) were to hold, then
\begin{align*}
e_1 -\mu+ 
\overline{e_2-\mu} = \frac{\textup{Im} \tau}{\pi}\left( | e_1-\mu|^2 +(e_3-\mu) \overline{e_2 - e_1}\right).
\end{align*}
Using the identities above together with $|\tau|=1$, this equation can be rewritten as
\begin{align}\label{eqeqdonghwi}
2\textup{Re}\tau\cdot(\tau e_1-\tau\mu-\pi i)+2\pi\textup{Im}\tau=\frac{\textup{Im}\tau}{\pi}\left(|\tau e_1-\tau\mu|^2+(\tau e_3-\tau\mu)(\tau e_1-\tau e_2)\right).
\end{align}
The conjugation identities further give
\begin{align*}
\overline{\tau e_2}=\tau e_1,\quad \tau e_3\in\mathbb{R},\quad \tau\mu+\pi i\in \mathbb{R}.
\end{align*}
Since $e_1+e_2+e_3=0$, we also have 
\begin{align*}
\tau e_3=-2\textup{Re}(\tau e_1).
\end{align*}
Moreover, $e_1\neq e_2$, and hence 
\begin{align*}
\textup{Im}(\tau e_1)\neq 0.
\end{align*}
We now compare the real and imaginary parts of (\ref{eqeqdonghwi}). Comparing imaginary parts gives
\begin{align*}
\tau\mu+\pi i=-2\textup{Re}(\tau e_1)-\frac{\pi\textup{Re}\tau}{\textup{Im}\tau},
\end{align*}
where we used $\tau e_3=-2\textup{Re}(\tau e_1)$. Substituting this identity into (\ref{eqeqdonghwi}) and comparing the real parts, we obtain
\begin{align*}
&2\textup{Re}\tau\left(3\textup{Re}(\tau e_1)+\frac{\pi\textup{Re}\tau}{\textup{Im}\tau}\right)+2\pi\textup{Im}\tau\\
&=\frac{\textup{Im}\tau}{\pi}\left(\left(3\textup{Re}(\tau e_1)+\frac{\pi\textup{Re}\tau}{\textup{Im}\tau}\right)^2+\textup{Im}(\tau e_1)^2+\pi^2\right).
\end{align*}
Using $|\tau|=1$, this simplifies to
\begin{align*}
9\textup{Re}(\tau e_1)^2+\textup{Im}(\tau e_1)^2=\left(\frac{\pi}{\textup{Im}\tau}\right)^2.
\end{align*}

On the other hand, since $\tau e_3=-2\textup{Re}(\tau e_1)$, we have
\begin{align*}
\tau e_1-\tau e_3=3\textup{Re}(\tau e_1)+i\textup{Im}(\tau e_1).
\end{align*}
Therefore, 
\begin{align*}
|e_1-e_3|=\frac{\pi}{\textup{Im}\tau}.
\end{align*}
From (\ref{thetae1}), together with $\vartheta_3^4=\vartheta_2^4+\vartheta_4^4$, we obtain
\begin{align*}
e_1-e_3=\pi^2\vartheta_4^4.
\end{align*}
Consequently, (\ref{donghwi}) would imply 
\begin{align*}
\pi\textup{Im}\tau\left|\vartheta_4(\tau)\right|^4=1.
\end{align*}

However, for $q=e^{i\pi \tau}$, we have
\begin{align*}
|q|=e^{-\pi\textup{Im}\tau}\leq e^{-\pi\frac{\sqrt{3}}{2}}<\frac{1}{11}.
\end{align*}
Since 
\begin{align*}
\vartheta_4(\tau)=1 + 2 \sum_{n=1}^\infty (-1)^n q^{n^2},
\end{align*}
it follows that
\begin{align*}
\left|\vartheta_4(\tau)\right|\geq 1-2\sum_{n=1}^\infty|q|^{n^2}\geq 1-2\sum_{n=1}^\infty |q|^n=\frac{1-3|q|}{1-|q|}>\frac{4}{5}.
\end{align*}
Hence,
\begin{align*}
\pi\textup{Im}\tau\left|\vartheta_4(\tau)\right|^4\geq\pi\frac{\sqrt{3}}{2}\left(\frac{4}{5}\right)^4>1
\end{align*}
for $|q|\leq e^{-\pi\frac{\sqrt{3}}{2}}$. This is a contradiction.

The same argument applies to the case $\wp_1=e_2$ and $\wp_3=e_1$. Therefore, (\ref{donghwi}) cannot hold on the open circular arc, and the lemma follows.
\end{proof}

\begin{lemma}\label{lemCC}
Suppose that
\begin{align*}
\tau \in F \cap \{ z \in \mathbb{C}\ |\ \mathrm{Re}(z) = 0 \}.
\end{align*}
If two half-period values $\wp_1,\wp_3$ satisfy (\ref{ponepthree}), then (\ref{donghwi}) cannot hold.
\end{lemma}
\begin{proof}
Since $q=e^{i\pi\tau}\in\mathbb{R}$, the identities in (\ref{thetae1}) show that $e_1,e_2,e_3\in\mathbb{R}$. Then (\ref{ponepthree}), together with $e_1+e_2+e_3=0$ and the distinctness of the half-period values, implies that 
\begin{align*}
\{e_1,e_2,e_3\}=\{t,-t,0\}
\end{align*}
for some nonzero $t\in\mathbb{R}$. From (\ref{thetae1}), this can occur only when
\begin{align*}
e_3=0,\quad e_1=-e_2.
\end{align*}
Thus $g_3(\tau)=0$, and hence $j(\tau)=1728$. Since $\tau\in F$ lies on the imaginary axis, we conclude that $\tau=i$, and consequently $\wp_1=-\wp_3=\pm e_1(i)$.

At $\tau=i$, we have
\begin{align*}
\mu =- \frac{\pi^2}{3}E_2(i)=-\pi.
\end{align*}
Substituting this into (\ref{donghwi}) gives
\begin{align*}
2\pi = \frac{1}{\pi} \left( \left(\pi \pm e_1(i)\right)^2 \mp 2 e_1(i) \pi \right),
\end{align*}
which reduces to
\begin{align}\label{false2}
e_1(i)=\pi.
\end{align}
However, using the classical identity
\begin{align*}
e_1(i) = \frac{\Gamma\left( \frac{1}{4} \right)^4}{8\pi} \approx 6.875,
\end{align*}
we see that (\ref{false2}) is false. This completes the proof.
\end{proof}

We are now ready to prove Lemma \ref{lemmadaimma}.
\begin{proof}[Proof of Lemma \ref{lemmadaimma}]
It suffices to show that the two conditions in Lemma \ref{rank} cannot hold simultaneously. Suppose that 
\begin{align*}
|\wp_1| = |\wp_3|.
\end{align*}
By Lemma \ref{jreal}, (\ref{ponepthree}) implies that the $j$-invariant associated with the Weierstrass $\wp$-function is real. Remark \ref{remarkdaimma} then shows that $\tau$ lies either on $\partial F$ or on the imaginary axis in $F$. Lemma \ref{lemAA}, Lemma \ref{lemBB}, and Lemma \ref{lemCC}, which treat the vertical boundary, the circular boundary, and the imaginary axis, respectively, show that
\begin{align*}
\wp_1 -\mu+ 
\overline{\wp_3-\mu} = \frac{\textup{Im} \tau}{\pi}\left( | \wp_1-\mu|^2 +(\wp_4-\mu) \overline{\wp_3 - \wp_1}\right)
\end{align*}
cannot hold in any of these cases. Therefore, the two conditions in Lemma \ref{rank} cannot hold simultaneously. This proves the lemma.
\end{proof}

\section{Symmetry Group}\label{SecSec4}
\setcounter{equation}{0}
Let $S_1\subset\mathbb{R}^4$ be a complete, oriented, embedded minimal surface with finite total curvature and genus $1$. Suppose that $S_1$ has three embedded planar ends whose asymptotic planes are parallel to 
\begin{align*}
Q_1:= \left\{(z, 0)\in\mathbb{C}^2\ |\ z\in\mathbb{C}\right\},\quad Q_3:=\left\{(0,z)\in\mathbb{C}^2\ |\ z\in\mathbb{C}\right\},
\end{align*}
and
\begin{align*}
Q_2:=\left\{\left(z, \bar{z}\right)\in\mathbb{C}^2\ |\ z\in\mathbb{C}\right\}.
\end{align*}
We note that $Q_1$ and $Q_3$ are the two asymptotic planes of the Lagrangian catenoid, while $Q_2$ is the plane containing its waist circle.

A symmetry of $S_1$ is a rigid motion $R=(A,b)$ of $\mathbb{R}^4$ such that $R(S_1)=S_1$. Here, $A\in O(4)$ and $b\in\mathbb{R}^4$, and $R$ acts on $\mathbb{R}^4$ by $R(x)=Ax+b$. We denote the group of all symmetries of $S_1$ by $\mathfrak{Sym}(S_1)$.

In this section, we prove that if $\left|\mathfrak{Sym}(S_1)\right|\geq 8$, then 
\begin{align*}
\left|\mathfrak{Sym}(S_1)\right|=8\quad \text{or}\quad 12.
\end{align*}

\subsection{Symmetries preserving all three ends} 
A symmetry of $S_1$ induces a permutation on the set of three planes $\{Q_1, Q_2, Q_3\}$. Define
\begin{align*}
\Theta_{V,W}:=\sup\left\{\langle v, w\rangle\ |\ v\in V,\ w\in W,\ \langle v,v\rangle=\langle w,w\rangle=1\right\},
\end{align*}
for $2$-dimensional $\mathbb{R}$-vector subspaces $V, W\subset\mathbb{R}^4$, where $\langle\cdot,\cdot\rangle$ denotes the Euclidean inner product on $\mathbb{R}^4$. It is clear that $\Theta_{V,W}=\Theta_{W,V}$, and for every $A\in O(4)$, 
\begin{align*}
\Theta_{V,W}=\Theta_{AV, AW}.
\end{align*}

A direct computation gives
\begin{align*}
\Theta_{Q_1,Q_2}=\frac{1}{\sqrt{2}},\quad \Theta_{Q_2,Q_3}=\frac{1}{\sqrt{2}},\quad \Theta_{Q_1,Q_3}=0.
\end{align*}
Hence, the invariance of $\Theta$ under symmetries implies that the only possible permutations induced by a symmetry $(A,b)\in\mathfrak{Sym}(S_1)$ are the identity permutation,
\begin{align*}
Q_1\stackrel{A}{\to}Q_1,\quad Q_2\stackrel{A}{\to}Q_2,\quad Q_3\stackrel{A}{\to}Q_3,
\end{align*}
and the permutation interchanging $Q_1$ and $Q_3$,
\begin{align*} 
Q_1\stackrel{A}{\to}Q_3,\quad Q_2\stackrel{A}{\to}Q_2,\quad Q_3\stackrel{A}{\to}Q_1. 
\end{align*}
Let $\mathfrak{Sym}^{id}(S_1)$ denote the set of symmetries that induce the identity permutation. It is a subgroup of $\mathfrak{Sym}(S_1)$, and since there are at most two possible induced permutations, its index is at most $2$.

Let $(A,b)\in\mathfrak{Sym}^{id}(S_1)$. Using the $A$-invariance of the subspaces $Q_1,Q_2,Q_3$, it follows that $A$ has the form
\begin{align}\label{AAform}
\begin{pmatrix}
\mathcal{S}_{\lambda} & O\\
O & \mathcal{S}_{-\lambda}
\end{pmatrix}\quad\text{or}\quad
\begin{pmatrix}
\mathcal{T}_{\lambda} & O\\
O & \mathcal{T}_{-\lambda}
\end{pmatrix},\quad \lambda\in\mathbb{R},
\end{align}
with respect to the standard basis of $\mathbb{R}^4$, where
\begin{align*}
\mathcal{S}_{\lambda}:=
\begin{pmatrix}
\cos\lambda & -\sin\lambda\\
\sin\lambda & \cos\lambda
\end{pmatrix},\quad 
\mathcal{T}_{\lambda}:=
\begin{pmatrix}
\cos\lambda & -\sin\lambda\\
-\sin\lambda & -\cos\lambda
\end{pmatrix}.
\end{align*}
Moreover, since
\begin{align*}
\begin{pmatrix}
O& I\\ I& O
\end{pmatrix}
\end{align*}
interchanges $Q_1$ and $Q_3$, the linear part of any symmetry $(A,b)\in\mathfrak{Sym}(S_1)\setminus\mathfrak{Sym}^{id}(S_1)$ can be written as
\begin{align}\label{InvAAform}
A=\begin{pmatrix}
O & \mathcal{S}_{-\lambda}\\
\mathcal{S}_{\lambda} & O
\end{pmatrix}\quad\text{or}\quad
\begin{pmatrix}
O & \mathcal{T}_{-\lambda}\\
\mathcal{T}_{\lambda} & O
\end{pmatrix},\quad\lambda\in\mathbb{R}.
\end{align}

\subsection{Orientation-preserving symmetries}
Suppose further that $S_1\subset \mathbb{R}^4$ is represented by an embedding $X:\Sigma_1\setminus\{q_1,q_2,q_3\}\to\mathbb{R}^4$ via the generalized Weierstrass representation (\ref{gwr}), where each puncture $q_j$ corresponds to an embedded planar end of $S_1$ parallel to $Q_j$.

Motivated by \cite{HM2}, let $\mathcal{H}$ and $\mathcal{G}$ denote the subgroups of orientation-preserving symmetries in $\mathfrak{Sym}^{id}(S_1)$ and $\mathfrak{Sym}(S_1)$, respectively. Then $\mathcal{H}=\mathcal{G}\cap\mathfrak{Sym}^{id}(S_1)$. Since the index of $\mathfrak{Sym}^{id}(S_1)$ in $\mathfrak{Sym}(S_1)$ is at most $2$, it follows that $\mathcal{H}$ is a subgroup of $\mathcal{G}$ of index at most $2$. Moreover, the composition of two orientation-reversing symmetries is orientation-preserving, which implies that $\mathcal{G}$ is a subgroup of $\mathfrak{Sym}(S_1)$ of index at most $2$. Therefore, 
\begin{align}\label{indexgroup}
\mathcal{H}\subseteq\mathcal{G}\subseteq\mathfrak{Sym}(S_1),\quad \left[\mathcal{G}:\mathcal{H}\right]\leq 2,\quad\left[\mathfrak{Sym}(S_1):\mathcal{G}\right]\leq 2.
\end{align}
\begin{lemma}\label{fincyc}
$\mathcal{H}$ is a finite cyclic group.
\end{lemma}
Before proving the lemma, we first show that each element $(A,b)\in\mathcal{H}$ induces a biholomorphism of $\Sigma_1$. Since $S_1$ is given by an embedding $X:\Sigma_1\setminus\{q_1,q_2,q_3\}\to\mathbb{R}^4$, the action of $(A,b)$ on $S_1$ induces a conformal diffeomorphism
\begin{align*}
\mu_{(A,b)}:\Sigma_1\setminus\{q_1,q_2,q_3\}\to \Sigma_1\setminus\{q_1,q_2,q_3\}
\end{align*}
such that $(A,b)\circ X=X\circ\mu_{(A,b)}$. As $(A,b)\in\mathcal{H}$ preserves each end of $S_1$, the map $\mu_{(A,b)}$ extends to a homeomorphism $\tilde{\mu}_{(A,b)}$ of $\Sigma_1$ by setting $\tilde{\mu}_{(A,b)}(q_j)=q_j$ for every $1\leq j\leq 3$. Riemann's removable singularity theorem then shows that $\tilde{\mu}_{(A,b)}$ is a biholomorphism.

\begin{proof}[Proof of Lemma \ref{fincyc}]
The biholomorphisms of $\Sigma_1$ induced by the elements of $\mathcal{H}$ fix the three points $q_1,q_2,q_3$. It is well known that the group of biholomorphisms of $\Sigma_1$ fixing three points is finite. This shows that $\mathcal{H}$ is a finite group.

Moreover, since elements of $\mathcal{H}$ are orientation-preserving, it follows from (\ref{AAform}) that if $(A,b)\in\mathcal{H}$, then $A$ has the form 
\begin{align*}
\begin{pmatrix}
\mathcal{S}_{\lambda} & O\\
O & \mathcal{S}_{-\lambda}
\end{pmatrix}
\end{align*}
for some $\lambda\in\mathbb{R}$. Consider the group homomorphism
\begin{align*}
\mathcal{H}\to S^1:=\left\{e^{i\theta}\in\mathbb{C}\ \big{|}\ \theta\in\mathbb{R}\right\}
\end{align*}
defined by $(A,b)\mapsto e^{i\lambda}$. Since $S_1$ is not totally geodesic and has finite total curvature, $\mathfrak{Sym}(S_1)$ contains no nontrivial translations. Hence, the above homomorphism is injective, and we may regard $\mathcal{H}$ as a subgroup of $S^1$. It follows that $\mathcal{H}$, being a finite subgroup of $S^1$, is cyclic.
\end{proof}
 
\subsection{Order of the symmetry group}\label{subsecordsym}
Suppose that $\left|\mathfrak{Sym}(S_1)\right|\geq 8$. It follows from (\ref{indexgroup}) that $|\mathcal{H}|\geq 2$. Let $(A_{\mathcal{H}},b_{\mathcal{H}})\in\mathcal{H}$ be a generator. Then it is nontrivial, and $\mathcal{H}$ acts freely on $\mathbb{R}^4\setminus\{p_0\}$, where $p_0:=(I-A_{\mathcal{H}})^{-1}b_{\mathcal{H}}$.

As shown in the previous subsection, $(A_{\mathcal{H}},b_{\mathcal{H}})$ induces a biholomorphism $\tilde{\mu}_{(A_{\mathcal{H}},b_{\mathcal{H}})}$ of $\Sigma_1$, which we denote by $\tilde{\mu}_{\mathcal{H}}$ for simplicity. Following \cite{HM2}, consider the quotient map
\begin{align*}
\mathcal{Q}_{\mathcal{H}}:\Sigma_1\to \Sigma_1/\langle \tilde{\mu}_{\mathcal{H}}\rangle.
\end{align*}
The branch points of $\mathcal{Q}_{\mathcal{H}}$ are precisely the fixed points of $\tilde{\mu}_{\mathcal{H}}$ given by
\begin{align*}
\{q_1,q_2,q_3\}\cup X^{-1}(p_0).
\end{align*}
Let $n$ denote the number of branch points of $\mathcal{Q}_{\mathcal{H}}$. Then $n\geq 3$, and the branching order at each branch point is $|\mathcal{H}|-1$.

The Riemann--Hurwitz formula gives
\begin{align*}
|\mathcal{H}|\chi\left(\Sigma_1/\langle \tilde{\mu}_{\mathcal{H}}\rangle\right)=\chi(\Sigma_1)+n(|\mathcal{H}|-1)\geq 3>0.
\end{align*}
Thus, $\Sigma_1/\langle \tilde{\mu}_{\mathcal{H}}\rangle$ is a sphere and hence has Euler characteristic $2$. Substituting this into the above formula gives
\begin{align*}
|\mathcal{H}|=1+\frac{2}{n-2}.
\end{align*}
Since $|\mathcal{H}|\geq 2$, we have $n\leq 4$. Hence, $n=3$ or $4$, corresponding to $|\mathcal{H}|=3$ or $2$, respectively. In particular, $|\mathcal{H}|\leq 3$.

As $\left|\mathfrak{Sym}(S_1)\right|\geq 8>2|\mathcal{H}|$, both indices in (\ref{indexgroup}) must be equal to $2$. Therefore,
\begin{align*}
\left[\mathcal{G}:\mathcal{H}\right]=2,\quad\left[\mathfrak{Sym}(S_1):\mathcal{G}\right]=2
\end{align*}
and consequently
\begin{align*}
\left|\mathfrak{Sym}(S_1)\right|=4|\mathcal{H}|.
\end{align*}

The above discussion can be summarized as follows.
\begin{prop}\label{proptwocase}
Suppose that $\left|\mathfrak{Sym}(S_1)\right|\geq 8$. Then one of the following holds:
\begin{itemize}
\item[(1)] $|\mathcal{H}|=3$, and the fixed points of $\tilde{\mu}_{\mathcal{H}}$ are $\{q_1, q_2, q_3\}$.
\item[(2)] $|\mathcal{H}|=2$, and the fixed points of $\tilde{\mu}_{\mathcal{H}}$ are $\{q_0, q_1, q_2, q_3\}$, where $q_0\in\Sigma_1$ is the unique point such that $X(q_0)=p_0\left(=(I-A_{\mathcal{H}})^{-1}b_{\mathcal{H}}\right)$.
\end{itemize}
In both cases, $\left|\mathfrak{Sym}(S_1)\right|=4|\mathcal{H}|$. 
\end{prop}
In particular, if $\left|\mathfrak{Sym}(S_1)\right|\geq 8$, then $\left|\mathfrak{Sym}(S_1)\right|=8$ or $12$.

\section{Classification of Embedded Minimal Tori with $\left|\mathfrak{Sym}(S_1)\right|\geq8$}\label{sec5}
\setcounter{equation}{0}
In this section, we continue to consider a complete, oriented, embedded minimal surface $S_1\subset\mathbb{R}^4$ with finite total curvature and genus $1$, having three embedded planar ends whose asymptotic planes are parallel to 
\begin{align*}
Q_1= \left\{(z, 0)\in\mathbb{C}^2\ |\ z\in\mathbb{C}\right\},\quad Q_3=\left\{(0,z)\in\mathbb{C}^2\ |\ z\in\mathbb{C}\right\},
\end{align*}
and
\begin{align*}
Q_2=\left\{\left(z, \bar{z}\right)\in\mathbb{C}^2\ |\ z\in\mathbb{C}\right\}.
\end{align*}
Here, $Q_1$ and $Q_3$ are the two asymptotic planes of the Lagrangian catenoid, while $Q_2$ is the plane containing its waist circle. By Proposition \ref{proptwocase}, if $\left|\mathfrak{Sym}(S_1)\right|\geq 8$, then $\left|\mathfrak{Sym}(S_1)\right|=8$ or $12$. We classify the surfaces corresponding to each of these two cases.

Suppose that $S_1$ is given by an embedding $X:\Sigma_1\setminus\{q_1,q_2,q_3\}\to\mathbb{R}^4$ via the generalized Weierstrass representation (\ref{gwr}), where each $q_j$ corresponds to an embedded planar end of $S_1$ parallel to $Q_j$. By Theorem \ref{thm1}, $S_1$ is $J$-holomorphic for some almost complex structure $J$. In the proof of Theorem \ref{thm1}, we showed that the associated Weierstrass data satisfy 
\begin{align*}
\phi_2 =i \phi_1,\quad \phi_4 = -i \phi_3. 
\end{align*}
Together with the vanishing real period conditions, these identities imply that all periods of $\phi_1$ and $\phi_3$ vanish. Hence, there exist meromorphic functions $f$ and $g$ on $\Sigma_1$ such that  
\begin{align*}
df=\phi_1,\quad dg=\phi_3.
\end{align*}
By (\ref{gwr}), up to translation, $S_1$ is represented by
\begin{align*}
\left(\text{Re}f, -\text{Im}f, \text{Re}g, \text{Im}g\right)
\end{align*}
on $\Sigma_1\setminus\{q_1, q_2, q_3\}$. Under the identification $\mathbb{R}^4\simeq\mathbb{C}^2$, this is precisely $(\overline{f}, g)$.

As a preliminary step toward the classification, we show that, provided $\left|\mathfrak{Sym}(S_1)\right|\geq 8$, the asymptotic planes of $S_1$ are not just parallel to $Q_1,Q_2,Q_3$, but in fact coincide with them up to translation.
\begin{lemma}\label{asymptotichada}
If $\left|\mathfrak{Sym}(S_1)\right|\geq 8$, then $S_1$ has, up to translation, the same asymptotic planes as $Q_1, Q_2, Q_3$.
\end{lemma}
\begin{proof}
We may write the three asymptotic planes of $S_1$ as
\begin{align*}
\{(\alpha, w)\in\mathbb{C}^2\ |\ w\in\mathbb{C}\},\quad \{(z,\beta)\in\mathbb{C}^2\ |\ z\in\mathbb{C}\},
\end{align*}
and a plane $Q$ parallel to $Q_2$, for some $\alpha,\beta\in\mathbb{C}$. We claim that $Q$ passes through $(\alpha,\beta)$.

Let $\mathcal{H}$ be the group of orientation-preserving symmetries of $S_1$ preserving each end, and let $(A_{\mathcal{H}},b_{\mathcal{H}})$ be a generator of $\mathcal{H}$. As discussed in Section \ref{subsecordsym}, if $\left|\mathfrak{Sym}(S_1)\right|\geq 8$, then $|\mathcal{H}|\geq 2$, and hence $(A_{\mathcal{H}},b_{\mathcal{H}})$ is nontrivial. Since the first two asymptotic planes are preserved by this symmetry, their intersection point $(\alpha,\beta)$ is fixed by $(A_{\mathcal{H}},b_{\mathcal{H}})$.

Let $(\alpha', \beta')$ be the orthogonal projection of $(\alpha, \beta)$ onto $Q$. Since $(A_{\mathcal{H}},b_{\mathcal{H}})$ is an isometry preserving $Q$ and fixing $(\alpha,\beta)$, the uniqueness of the orthogonal projection implies that $(\alpha',\beta')$ is also fixed. Therefore,
\begin{align*}
A_{\mathcal{H}}\left((\alpha',\beta')-(\alpha,\beta)\right)=(\alpha',\beta')-(\alpha,\beta).
\end{align*}
Since $A_{\mathcal{H}}$ has no nonzero fixed vectors, we obtain
\begin{align*}
(\alpha',\beta')=(\alpha,\beta).
\end{align*}
Hence, $(\alpha,\beta)\in Q$.

Thus, all three asymptotic planes meet at the common point $(\alpha,\beta)$. After translating this point to the origin, the three asymptotic planes coincide with $Q_1, Q_2, Q_3$. This proves the lemma.
\end{proof}

\subsection{The case $\left|\mathfrak{Sym}(S_1)\right|=8$}
For distinct complex numbers $\lambda_1,\lambda_2,\lambda_3,\lambda_4\in\mathbb{C}$, we denote
\begin{align*}
\Sigma_{\boldsymbol{\lambda}}:=\left\{ (x, y) \in \left(\mathbb{C}\cup\{\infty\}\right)^2\ \middle|\ y^2 = (x - \lambda_1)(x - \lambda_2)(x - \lambda_3)(x - \lambda_4) \right\},
\end{align*}
where $\boldsymbol{\lambda}$ denotes the quadruple $(\lambda_1,\lambda_2,\lambda_3,\lambda_4)$. We note that every compact Riemann surface of genus $1$ is conformally equivalent to some $\Sigma_{\boldsymbol{\lambda}}$.

Define meromorphic functions $f_{\boldsymbol{\lambda}}$ and $g_{\boldsymbol{\lambda}}$ on $\Sigma_{\boldsymbol{\lambda}}$ by
\begin{align*}
f_{\boldsymbol{\lambda}}= \frac{y}{(x - \lambda_1)(x - \lambda_2)}, \quad
g_{\boldsymbol{\lambda}} = \frac{\lambda_2 - \lambda_3}{\lambda_2 - \lambda_1} \cdot \frac{y}{(x - \lambda_2)(x - \lambda_3)}.
\end{align*}
Then $f_{\boldsymbol{\lambda}}$ has simple poles at $(\lambda_1,0)$ and $(\lambda_2,0)$, and simple zeros at $(\lambda_3,0)$ and $(\lambda_4,0)$. Similarly, $g_{\boldsymbol{\lambda}}$ has simple poles at $(\lambda_2,0)$ and $(\lambda_3,0)$, and simple zeros at $(\lambda_1,0)$ and $(\lambda_4,0)$.

In \cite[Section 5]{JL}, it was shown that the map
\begin{align*}
(f_{\boldsymbol{\lambda}},g_{\boldsymbol{\lambda}}):\Sigma_{\boldsymbol{\lambda}}\setminus\{(\lambda_1,0),(\lambda_2,0),(\lambda_3,0)\}\to\mathbb{C}^2
\end{align*}
is a minimal embedding with three embedded planar ends. Its asymptotic planes are $Q_1$, $Q_3$, and
\begin{align*}
\{(z,w)\in\mathbb{C}^2\ |\ z=w\},
\end{align*}
which is the reflection of $Q_2$ across the hyperplane $\{x_2=0\}$. We denote its image by
\begin{align*}
S_{\boldsymbol{\lambda}}.
\end{align*}

\begin{lemma}\label{eightsymsym}
The surface $S_{\boldsymbol{\lambda}}$ satisfies
\begin{align*}
\left|\mathfrak{Sym}\left(S_{\boldsymbol{\lambda}}\right)\right|=8
\end{align*}
if and only if
\begin{align*}
\lambda_4 (2 \lambda_2 - \lambda_1 - \lambda_3)= \lambda_1 \lambda_2 + \lambda_2 \lambda_3 - 2 \lambda_1 \lambda_3.
\end{align*}
In this case, $\Sigma_{\boldsymbol{\lambda}}$ is conformally equivalent to the square torus $\mathbb{C}/\Lambda(1,i)$. Moreover, the corresponding surface is unique up to scaling and a unitary transformation.
\end{lemma}
\begin{proof}
Suppose that $S_{\boldsymbol{\lambda}}$ has $8$ symmetries. By reflecting $S_{\boldsymbol{\lambda}}$ across the hyperplane $\{x_2=0\}$, we may apply (\ref{AAform}) and (\ref{InvAAform}). Reflecting back, we see that, for some $\theta_0\in\mathbb{R}$, the embedding is invariant under the two transformations
\begin{align}\label{symfg}
(f_{\boldsymbol{\lambda}},g_{\boldsymbol{\lambda}})\mapsto (ig_{\boldsymbol{\lambda}}, if_{\boldsymbol{\lambda}})
\end{align}
and
\begin{align}\label{symfgfgfg}
(f_{\boldsymbol{\lambda}},g_{\boldsymbol{\lambda}})\mapsto (e^{i\theta_0}\overline{f_{\boldsymbol{\lambda}}}, e^{i\theta_0}\overline{g_{\boldsymbol{\lambda}}}).
\end{align}

We first consider the symmetry (\ref{symfg}). Since $S_{\boldsymbol{\lambda}}$ is embedded, this symmetry induces a diffeomorphism of $\Sigma_{\boldsymbol{\lambda}}\setminus\{(\lambda_1,0), (\lambda_2,0), (\lambda_3,0)\}$. Thus, for each $(x,y)\in\Sigma_{\boldsymbol{\lambda}}\setminus\{(\lambda_1,0), (\lambda_2,0), (\lambda_3,0)\}$, there exists 
\begin{align*}
\left(\tilde{x},\tilde{y}\right)=\left(\tilde{x}(x,y),\tilde{y}(x,y)\right)\in\Sigma_{\boldsymbol{\lambda}}\setminus\{(\lambda_1,0), (\lambda_2,0), (\lambda_3,0)\}
\end{align*}
such that 
\begin{align} \label{symfg1}
i \frac{\lambda_2-\lambda_3}{\lambda_2-\lambda_1} \cdot \frac{y}{(x - \lambda_2)(x - \lambda_3)}= \frac{\tilde{y}}{(\tilde{x} - \lambda_1)(\tilde{x} - \lambda_2)}
\end{align}
and
\begin{align}\label{symfg12}
i \frac{y}{(x - \lambda_1)(x - \lambda_2)}= \frac{\lambda_2-\lambda_3}{\lambda_2-\lambda_1} \cdot \frac{\tilde{y}}{(\tilde{x} - \lambda_2)(\tilde{x} - \lambda_3)}.
\end{align}
Dividing \eqref{symfg1} by (\ref{symfg12}), we obtain
\[
\frac{\lambda_2-\lambda_3}{\lambda_2-\lambda_1}\cdot\frac{ x - \lambda_1}{x - \lambda_3} = \frac{\lambda_2-\lambda_1}{\lambda_2-\lambda_3}\cdot\frac{\tilde{x} - \lambda_3}{\tilde{x} - \lambda_1}.
\]
Solving for \( \tilde{x} \) gives
\begin{align*}
\tilde{x}= \frac{\lambda_3 (\lambda_2 - \lambda_1)^2 (x - \lambda_3) - \lambda_1 (\lambda_2 - \lambda_3)^2 (x - \lambda_1)}{(\lambda_2 - \lambda_1)^2 (x - \lambda_3) - (\lambda_2 - \lambda_3)^2 (x - \lambda_1)},
\end{align*}
or equivalently,
\begin{align} \label{xtilde}
\tilde{x}= \frac{(\lambda_2^2 - \lambda_1 \lambda_3)x - \lambda_2(\lambda_1 \lambda_2 + \lambda_2 \lambda_3 - 2 \lambda_1 \lambda_3)}{(2\lambda_2 - \lambda_1 - \lambda_3)x - (\lambda_2^2 - \lambda_1 \lambda_3)}.
\end{align}
Substituting (\ref{xtilde}) into \eqref{symfg1} yields
\begin{align} \label{ytilde}
\tilde{y}
= -i \cdot \frac{(\lambda_1 - \lambda_2)^2 (\lambda_2 - \lambda_3)^2 y}{\left[(2 \lambda_2 - \lambda_1 - \lambda_3)x - (\lambda_2^2 - \lambda_1 \lambda_3)\right]^2}.
\end{align}

We now substitute \eqref{xtilde} and \eqref{ytilde} into the defining equation
\begin{align*}
\tilde{y}^2 = (\tilde{x} - \lambda_1)(\tilde{x} - \lambda_2)(\tilde{x} - \lambda_3)(\tilde{x} - \lambda_4),
\end{align*}
and eliminate $y$ using 
\begin{align*}
y^2= (x - \lambda_1)(x - \lambda_2)(x - \lambda_3)(x - \lambda_4).
\end{align*}
Then we obtain a rational expression in $x$. After clearing denominators, the resulting expression must vanish identically as a polynomial in $x$, which leads to 
\begin{align}\label{lamb4}
\lambda_4 (2\lambda_2 - \lambda_1 - \lambda_3) = \lambda_1 \lambda_2 + \lambda_2 \lambda_3 - 2\lambda_1 \lambda_3.
\end{align}
Conversely, if this identity holds, then (\ref{xtilde}) and (\ref{ytilde}) define a diffeomorphism satisfying (\ref{symfg1}) and (\ref{symfg12}). Hence $S_{\boldsymbol{\lambda}}$ admits a symmetry of the form (\ref{symfg}). 

We next show that any choice of $\lambda_1, \lambda_2, \lambda_3, \lambda_4$ satisfying (\ref{lamb4}) also yields a symmetry of the form (\ref{symfgfgfg}). Set
\begin{align*}
\beta:=\frac{\lambda_2-\lambda_3}{\lambda_2-\lambda_1},
\end{align*}
and choose $\theta_0\in\mathbb{R}$ such that 
\begin{align}\label{hathathat}
e^{2i\theta_0}=\frac{\beta(1+\overline{\beta})(1-\beta)}{\overline{\beta}(1+\beta)(1-\overline{\beta})}.
\end{align}
Define
\begin{align}\label{hatx}
\hat{x}=\frac{\lambda_3 \overline{\beta} (\overline{x}-\overline{\lambda_1}) - \lambda_1 \beta (\overline{x} - \overline{\lambda_3} )}{ \overline{\beta} (\overline{x} - \overline{\lambda_1}) - \beta (\overline{x} - \overline{\lambda_3})}
\end{align}
and
\begin{align}\label{haty}
\hat{y}=e^{i\theta_0}\frac{(\hat{x}-\lambda_1)(\hat{x}-\lambda_2)}{(\overline{x}-\overline{\lambda_1})(\overline{x}-\overline{\lambda_2})}\overline{y}
\end{align}
for $(x,y)\in\Sigma_{\boldsymbol{\lambda}}$. We first verify that $(\hat{x}, \hat{y})\in\Sigma_{\boldsymbol{\lambda}}$. 

A direct computation shows that (\ref{hatx}) is equivalent to 
\begin{align}\label{hatxx}
\frac{1}{\overline{\beta}}\cdot\frac{\overline{x}-\overline{\lambda_3}}{\overline{x}-\overline{\lambda_1}}=\frac{1}{\beta}\cdot\frac{\hat{x}-\lambda_3}{\hat{x}-\lambda_1}.
\end{align} 
Moreover, from (\ref{hatx}), 
\begin{align*}
\hat{x}-\lambda_2=\frac{(\lambda_3-\lambda_2)\overline{\beta}(\overline{x}-\overline{\lambda_1})+(\lambda_2-\lambda_1)\beta(\overline{x}-\overline{\lambda_3})}{\overline{\beta} (\overline{x} - \overline{\lambda_1}) - \beta (\overline{x} - \overline{\lambda_3})},
\end{align*}
which can be rewritten as
\begin{align*}
\hat{x}-\lambda_2=\frac{\left((\lambda_3-\lambda_2)\overline{\beta}+(\lambda_2-\lambda_1)\beta\right)(\overline{x}-\overline{\lambda_2})}{\overline{\beta} (\overline{x} - \overline{\lambda_1}) - \beta(\overline{x} - \overline{\lambda_3})}.
\end{align*}
Hence
\begin{align}\label{hatxxx}
\frac{\hat{x}-\lambda_2}{\overline{x}-\overline{\lambda_2}}=\frac{(\lambda_3-\lambda_2)\overline{\beta}+(\lambda_2-\lambda_1)\beta}{\overline{\beta} (\overline{x} - \overline{\lambda_1}) - \beta(\overline{x} - \overline{\lambda_3})}.
\end{align}
By (\ref{lamb4}), 
\begin{align*}
\lambda_4=\frac{\beta\lambda_1+\lambda_3}{1+\beta}.
\end{align*}
Using this identity in (\ref{hatx}), we obtain
\begin{align*}
\hat{x}-\lambda_4=\frac{(\lambda_3-\lambda_4)\overline{\beta}(\overline{x}-\overline{\lambda_1})+(\lambda_4-\lambda_1)\beta(\overline{x}-\overline{\lambda_3})}{\overline{\beta} (\overline{x} - \overline{\lambda_1}) - \beta(\overline{x} - \overline{\lambda_3})},
\end{align*}
and hence
\begin{align*}
\hat{x}-\lambda_4=\frac{\beta(1+\overline{\beta})}{1+\beta}\cdot\frac{(\lambda_3-\lambda_1)(\overline{x}-\overline{\lambda_4})}{\overline{\beta} (\overline{x} - \overline{\lambda_1}) - \beta(\overline{x} - \overline{\lambda_3})}.
\end{align*}
Therefore,
\begin{align}\label{hatxxxx}
\frac{\hat{x}-\lambda_4}{\overline{x}-\overline{\lambda_4}}=\frac{\beta(1+\overline{\beta})}{1+\beta}\cdot\frac{\lambda_3-\lambda_1}{\overline{\beta} (\overline{x} - \overline{\lambda_1}) - \beta(\overline{x} - \overline{\lambda_3})}.
\end{align}
Squaring the expression for $\hat{y}$ in (\ref{haty}) and using the defining equation of $\Sigma_{\boldsymbol{\lambda}}$, we have
\begin{align*}
\hat{y}^2&=e^{2i\theta_0}\frac{(\hat{x}-\lambda_1)^2(\hat{x}-\lambda_2)^2(\overline{x}-\overline{\lambda_3})(\overline{x}-\overline{\lambda_4})}{(\overline{x}-\overline{\lambda_1})(\overline{x}-\overline{\lambda_2})}\\
&=e^{2i\theta_0}\frac{(\hat{x}-\lambda_1)(\overline{x}-\overline{\lambda_3})(\hat{x}-\lambda_2)(\overline{x}-\overline{\lambda_4})}{(\hat{x}-\lambda_3)(\overline{x}-\overline{\lambda_1})(\overline{x}-\overline{\lambda_2})(\hat{x}-\lambda_4)}(\hat{x}-\lambda_1)(\hat{x}-\lambda_2)(\hat{x}-\lambda_3)(\hat{x}-\lambda_4).
\end{align*}
Substituting (\ref{hathathat}), (\ref{hatxx}), (\ref{hatxxx}), and (\ref{hatxxxx}) into the above expression yields
\begin{align*}
&\frac{\hat{y}^2}{(\hat{x}-\lambda_1)(\hat{x}-\lambda_2)(\hat{x}-\lambda_3)(\hat{x}-\lambda_4)}\\
&\qquad=e^{2i\theta_0}\frac{\overline{\beta}}{\beta}\cdot\frac{1+\beta}{\beta(1+\overline{\beta})}\cdot\frac{(\lambda_3-\lambda_2)\overline{\beta}+(\lambda_2-\lambda_1)\beta}{\lambda_3-\lambda_1}\\
&\qquad=\frac{1-\beta}{\beta(1-\overline{\beta})}\cdot\left(\frac{\beta\overline{\beta}}{\beta-1}+\frac{\beta}{1-\beta}\right)\\
&\qquad=1,
\end{align*}
Thus, $(\hat{x}, \hat{y})\in\Sigma_{\boldsymbol{\lambda}}$. 

Finally, it is straightforward to verify that
\begin{align*}
f_{\boldsymbol{\lambda}}(\hat{x},\hat{y})=e^{i\theta_0}\overline{f_{\boldsymbol{\lambda}}(x,y)},\quad g_{\boldsymbol{\lambda}}(\hat{x},\hat{y})=e^{i\theta_0}\overline{g_{\boldsymbol{\lambda}}(x,y)}.
\end{align*}
Moreover, (\ref{haty}) and (\ref{hatxx}) show that the map $(x,y)\mapsto(\hat{x},\hat{y})$ is a diffeomorphism. Therefore, $S_{\boldsymbol{\lambda}}$ admits a symmetry of the form (\ref{symfgfgfg}).

It remains to identify the conformal type of $\Sigma_{\boldsymbol{\lambda}}$. There is a unique M\"obius transformation of the Riemann sphere sending $\lambda_1,\lambda_2,\lambda_3$ to $0, 1,\infty$, respectively, namely
$$z\mapsto \frac{\lambda_2 - \lambda_3}{\lambda_2 - \lambda_1}\cdot \frac{z- \lambda_1}{z-\lambda_3}.$$
By (\ref{lamb4}), this transformation sends $\lambda_4$ to $-1$. Hence, $\Sigma_{\boldsymbol{\lambda}}$ is conformally equivalent to $\mathbb{C}/\Lambda(1,i)$, where $\Lambda(1,i)$ is the square lattice. Moreover, it follows from this M\"obius transformation that the corresponding surface is unique up to scaling and a unitary transformation.
\end{proof}

We now show that the unique surface $S_{\boldsymbol{\lambda}}$ obtained in the preceding lemma characterizes the case $\left|\mathfrak{Sym}(S_1)\right|=8$.
\begin{thm} \label{thm2}
Let $S_1\subset\mathbb{R}^4$ be a complete, oriented, embedded minimal surface with finite total curvature and genus $1$, having three embedded planar ends whose asymptotic planes are parallel to those of the union of the Lagrangian catenoid $\Sigma_{LC}$ and the plane containing its waist circle. Suppose that $\left|\mathfrak{Sym}(S_1)\right|=8$. Then, up to scaling and rigid motion, $S_1$ coincides with the surface $S_{\boldsymbol{\lambda}}$ corresponding to the square lattice.
\end{thm}
\begin{proof}
As discussed at the beginning of Section \ref{sec5}, $S_1$ is represented by
\begin{align*}
(\overline{f},g):\Sigma_1\setminus\{q_1,q_2,q_3\}\to\mathbb{C}^2,
\end{align*}
for some meromorphic functions $f$ and $g$ on a closed Riemann surface $\Sigma_1$ of genus $1$, where each puncture $q_j$ corresponds to an end parallel to $Q_j$. By Lemma \ref{asymptotichada}, after a translation, $S_1$ has the same asymptotic planes as $Q_1,Q_2,Q_3$.

Suppose that $\Sigma_1$ is conformally equivalent to $\Sigma_{\boldsymbol{\lambda}}$ for some $\boldsymbol{\lambda}$, and identify the two Riemann surfaces. Consider the biholomorphism $\tilde{\mu}_{\mathcal{H}}$ introduced in Section \ref{subsecordsym}. Since $|\mathcal{H}|=2$ by Proposition \ref{proptwocase}, $\tilde{\mu}_{\mathcal{H}}$ is an involution. Moreover, by the same proposition, it has exactly four fixed points $\{q_0,q_1,q_2,q_3\}$, where $q_0\in\Sigma_1$ is the unique point satisfying $f(q_0)= g(q_0)=0$.

Since $\tilde{\mu}_{\mathcal{H}}$ is an involution with four fixed points on a genus $1$ Riemann surface, it must coincide with the hyperelliptic involution. Hence, the fixed points $q_0,q_1,q_2,q_3$ correspond to the branch points $(\lambda_j, 0)$ of $\Sigma_{\boldsymbol{\lambda}}$. After relabeling the indices, we may assume that $(\lambda_i,0)$ corresponds to $q_i$ for $1\leq i\leq 3$, while $(\lambda_4,0)$ corresponds to $q_0$.

Reflecting $S_1$ across the hyperplane $\{x_2=0\}\subset\mathbb{R}^4$, we may assume that it is represented by 
\begin{align*}
(f,g):\Sigma_1\setminus\{q_1,q_2,q_3\}\to\mathbb{C}^2.
\end{align*}
The asymptotic planes of the reflected surface are $Q_1$, $Q_3$, and 
\begin{align*}
\{(z,w)\in\mathbb{C}^2\ |\ z=w\},
\end{align*}
which is the reflection of $Q_2$ across $\{x_2=0\}$.

Since $S_1$ has embedded planar ends, $f$ has simple poles precisely at
\begin{align*}
q_1=(\lambda_1,0),\quad q_2=(\lambda_2,0),
\end{align*}
while $g$ has simple poles precisely at 
\begin{align*}
q_2=(\lambda_2,0),\quad q_3=(\lambda_3,0).
\end{align*}
Since the end at $q_3=(\lambda_3,0)$ is asymptotic to $Q_3$, we have $f(q_3)=0$, while the end at $q_1=(\lambda_1,0)$ being asymptotic to $Q_1$ gives $g(q_1)=0$. In addition, since $S_1$ passes through the origin at $q_0=(\lambda_4,0)$, both $f$ and $g$ vanish at $(\lambda_4,0)$.

Comparing zeros and poles, we conclude that
\begin{align*}
f = \alpha \cdot \frac{y}{(x - \lambda_1)(x - \lambda_2)}, \quad
g = \beta \cdot \frac{y}{(x - \lambda_2)(x - \lambda_3)},
\end{align*}
for some $\alpha,\beta\in\mathbb{C}\setminus\{0\}$. Moreover, since the end at $q_2=(\lambda_2,0)$ is asymptotic to the plane $\{(z,w)\in\mathbb{C}^2\ |\ z=w\}$, we must have
\begin{align*}
\left(\frac{g}{f}\right)(q_2)=\frac{\beta}{\alpha}\cdot\frac{\lambda_2-\lambda_1}{\lambda_2-\lambda_3}=1.
\end{align*} 
Thus,
\begin{align*}
\beta=\frac{\lambda_2-\lambda_3}{\lambda_2-\lambda_1}\alpha.
\end{align*}
It follows that
\begin{align*}
f=\alpha f_{\boldsymbol{\lambda}},\quad g=\alpha g_{\boldsymbol{\lambda}}.
\end{align*}
Therefore, the reflected surface is obtained from $S_{\boldsymbol{\lambda}}$ by scaling and a unitary transformation. The conclusion follows from Lemma \ref{eightsymsym}.
\end{proof}

\begin{rmk} \normalfont
If $\lambda_1,\lambda_2,\lambda_3,\lambda_4$ do not satisfy (\ref{lamb4}), then $S_{\boldsymbol{\lambda}}$ admits fewer than $8$ symmetries. Thus, although there exist many genus $1$ minimal embeddings whose asymptotic planes are parallel to those of the Lagrangian catenoid and the plane containing its waist circle, there is a unique one with $8$ symmetries, up to scaling and rigid motion, as characterized in the above theorem.
\end{rmk}

\subsection{The case $\left|\mathfrak{Sym}(S_1)\right|=12$}
We prove the following characterization theorem. Since the Costa--Hoffman--Meeks surface of genus $g$ has $4(g+1)$ symmetries, the case considered here exhibits a new high-codimensional phenomenon that does not occur in the setting of \cite{HM2}.

\begin{thm}\label{thm3}
Let $S_1\subset\mathbb{R}^4$ be a complete, oriented, embedded minimal surface with finite total curvature and genus $1$, having three embedded planar ends whose asymptotic planes are parallel to those of the union of the Lagrangian catenoid $\Sigma_{LC}$ and the plane containing its waist circle. Suppose that $\left|\mathfrak{Sym}(S_1)\right|=12$. Then, up to scaling and rigid motion, $S_1$ coincides with the holomorphic curve
\begin{align*}
\left\{ (z, w) \in \mathbb{C}^2 \ \middle|\ z w (z - w) = 1 \right\}.
\end{align*}
\end{thm}

\begin{proof} 
As discussed at the beginning of Section \ref{sec5}, $S_1$ is represented by an embedding 
\begin{align*}
(\overline{f},g):\Sigma_1\setminus\{q_1,q_2,q_3\}\to\mathbb{C}^2
\end{align*}
for some meromorphic functions $f$ and $g$ on $\Sigma_1$, where each puncture $q_j$ corresponds to an end parallel to $Q_j$. By Lemma \ref{asymptotichada}, we may further assume that the asymptotic planes of $S_1$ are precisely $Q_1,Q_2,Q_3$. 

Reflect $S_1$ across the hyperplane $\{x_2=0\}\subset\mathbb{R}^4$ and, by abuse of notation, continue to denote the reflected surface by $S_1$. Its asymptotic planes are then $Q_1$, $Q_3$, and
\begin{align*}
\{(z,w)\in\mathbb{C}^2\ |\ z=w\},
\end{align*}
which is the reflection of $Q_2$ across $\{x_2=0\}$. Moreover, the reflected surface is represented by 
\begin{align*}
(f,g):\Sigma_1\setminus\{q_1,q_2,q_3\}\to\mathbb{C}^2.
\end{align*}

As in the proof of Theorem \ref{thm2}, we examine the zeros and poles of $f$ and $g$. Since $S_1$ has embedded planar ends, $f$ has simple poles precisely at $q_1$ and $q_2$, while $g$ has simple poles precisely at $q_2$ and $q_3$. Moreover, the asymptotic behavior at $q_3$ and $q_1$ gives $f(q_3)=0$ and $g(q_1)=0$, respectively. 

We also consider the zeros and poles of $f-g$. From the information above, $f-g$ has simple poles precisely at $q_1$ and $q_3$. Since $S_1$ is asymptotic to $\{(z,w)\in\mathbb{C}^2\ |\ z=w\}$ at $q_2$, the function $f-g$ vanishes at $q_2$. Since a meromorphic function on a closed Riemann surface has the same number of zeros and poles, let $q^*$, $q^{**}$, and $q^{***}$ denote the remaining zeros of $f$, $g$, and $f-g$, respectively. These data are summarized in Table \ref{table:1}.

\begin{table}[h!]
\centering
\begin{tabular}{|>{\centering\arraybackslash}p{1cm}|>{\centering\arraybackslash}p{1.5cm}|>{\centering\arraybackslash}p{1.5cm}|} 
\hline
    & \textbf{Pole} & \textbf{Zero} \\ \hline
\( f \) & \( q_1, q_2 \) & \( q_3, q^* \) \\ \hline
\( g \) & \( q_2, q_3 \) & \( q_1, q^{**} \) \\ \hline
\( f-g \) & \( q_1, q_3 \) & \( q_2, q^{***} \) \\ \hline
\end{tabular}
\caption{Poles and zeroes of \( f, g \) and \( f-g \).}
\label{table:1}
\end{table}   
 
We now show that \( q^* = q_3 \). Consider the biholomorphism $\tilde{\mu}_{\mathcal{H}}$ introduced in Section \ref{subsecordsym}. By Proposition \ref{proptwocase}, \( \tilde{\mu}_{\mathcal{H}} \) has order \( 3 \), and its fixed points are precisely \( q_1, q_2, q_3 \). Suppose, to the contrary, that \( q^* \neq q_3 \). Then
\[
f(\tilde{\mu}_{\mathcal{H}}(q^*)) = e^{i\lambda} f(q^*) = 0,
\]
which implies that \( \tilde{\mu}_{\mathcal{H}}(q^*) \) is also a zero of \( f \). Therefore, \( \tilde{\mu}_{\mathcal{H}}(q^*) \in \{ q^*,\, q_3 \} \). If \( \tilde{\mu}_{\mathcal{H}}(q^*) = q_3 \), then applying \( \tilde{\mu}_{\mathcal{H}}^3 = \mathrm{id} \), we obtain
\[
q^* = \tilde{\mu}_{\mathcal{H}}^3(q^*) = \tilde{\mu}_{\mathcal{H}}^2(q_3) = q_3,
\]
which contradicts our assumption that \( q^* \neq q_3 \). On the other hand, if \( \tilde{\mu}_{\mathcal{H}}(q^*) = q^* \), then $q^*$ is a fixed point of \( \tilde{\mu}_{\mathcal{H}} \), giving at least four fixed points, again a contradiction. Therefore \( q^* = q_3 \). The same argument shows that \( q^{**} = q_1 \) and \( q^{***} = q_2 \).

Now we consider the meromorphic function 
$$ fg(f-g)$$
on $\Sigma_1$. From Table $\ref{table:1}$, we deduce that this function has neither poles nor zeros. By Liouville's theorem, it must be a nonzero constant. After applying an appropriate scaling and rotation to the surface, we may assume that $$fg(f-g)=1.$$ Thus, $S_1\subset\left\{ (z, w) \in \mathbb{C}^2 \ \middle|\ z w (z - w) = 1 \right\}$. Denote
\begin{align}\label{eqofS}
\mathcal{S}:=\left\{ (z, w) \in \mathbb{C}^2 \ \middle|\ z w (z - w) = 1 \right\}.
\end{align}
It remains to verify that $\mathcal{S}$ is a complete, oriented, embedded minimal surface in $\mathbb{C}^2$ with finite total curvature and genus $1$, asymptotic to the three planes $Q_1$, $Q_3$, and $\{(z,w)\in\mathbb{C}^2\ |\ z=w\}$, and with $12$ symmetries. Once this is established, the inclusion above implies that $S_1$ coincides with $\mathcal{S}$.

From its defining equation, $\mathcal{S}$ is asymptotic to $Q_1$, $Q_3$, and $\{(z,w)\in\mathbb{C}^2\ |\ z=w\}$. Moreover, $\mathcal{S}$ is invariant under 
\begin{align} \label{sym12}
\left\{
\begin{aligned}
(z, w) &\mapsto \left(e^{i \frac{\pi}{3}} w,\, e^{i \frac{\pi}{3}} z\right), \\
(z, w) &\mapsto \left(\overline{z},\, \overline{w} \right).
\end{aligned}
\right.
\end{align}
It is straightforward to verify that the group generated by these transformations consists of exactly \( 12 \) elements. Let
\[
A_1(z,w)=\left(e^{i\pi/3}w,e^{i\pi/3}z\right),
\quad
A_2(z,w)=(\overline z,\overline w)
\]
be the symmetries in \eqref{sym12}. Define
\[
\mathcal S_F^\circ
:=
\left\{(z,w)\in\mathcal S\ \bigg{|}\ 
\operatorname{Im}\left(\frac{w}{z}\right)>0,\ |w|>|z|,\ -\frac{\pi}{3}<\arg z<\frac{\pi}{3}
\right\},
\]
and let $\mathcal S_F$ be its closure in $\mathcal S$.

Writing $\xi=\frac{w}{z}$, the symmetries $A_1$ and $A_2$ induce $\xi\mapsto\frac{1}{\xi}$ and $\xi\mapsto\overline\xi$, respectively. If $\xi$ is not a real number and does not have unit modulus, exactly one of
\[
\xi,\quad \frac{1}{\xi},\quad \overline{\xi},\quad \frac{1}{\overline{\xi}}
\]
satisfies $\operatorname{Im}\xi>0$ and $|\xi|>1$. For each such $\xi$, the defining equation \eqref{eqofS} of $\mathcal S$ gives
\[
z^3=\frac{1}{\xi(1-\xi)},\quad w=\xi z.
\]
If $z_0$ is one solution and $\omega=e^{2\pi i/3}$, the three corresponding points are
\[
(z_0,\xi z_0),\quad
(\omega z_0,\omega\xi z_0),\quad
(\omega^2z_0,\omega^2\xi z_0).
\]
Since
\[
A_1^2(z,w)=(\omega z,\omega w),
\]
this symmetry fixes $\xi$ and permutes these three points cyclically. Moreover, $\xi(1-\xi)\notin(-\infty,0]$ for
$\operatorname{Im}\xi>0$, so exactly one of the three solutions satisfies
\[
-\frac{\pi}{3}<\arg z<\frac{\pi}{3}.
\]
Hence each orbit away from the boundary curves meets $\mathcal S_F^\circ$ exactly once. Its $12$ symmetric copies under \eqref{sym12} are pairwise disjoint, and their closures cover the entire $\mathcal S$. Thus $\mathcal S_F$ is a fundamental domain for the group generated by $A_1$ and $A_2$.

Furthermore, the projection onto the first coordinate is injective on its interior: two distinct ratios corresponding to the same $z$ must be $\xi$ and $1-\xi$, which cannot both have positive imaginary parts. Hence its interior is a graph.
The three boundary arcs, corresponding to $\xi>1$, $\xi<-1$, and $|\xi|=1$ with $\operatorname{Im}\xi>0$, lie in fixed planes of reflections generated by
\eqref{sym12}. Since
\[
\frac{\partial}{\partial w}\bigl(zw(z-w)\bigr)
=z(z-2w)\neq0
\]
on $\mathcal S_F$, the graph extends smoothly across its boundary points. Its $12$ symmetric copies have disjoint interiors and meet along these boundaries, proving that $\mathcal S$ is embedded.

The projective closure of $\mathcal S$ is the nonsingular cubic, hence by the degree-genus formula (see \cite[Chapter~7.2.2]{D}),
\[
g = \frac{(3 - 1)(3 - 2)}{2} = 1.
\]
Therefore $\mathcal{S}$ has genus $1$. Finally applying the Gauss--Bonnet formula, we obtain
\[
\int_{\mathcal S}K\,dA=-4\pi(g+3-1)=-12\pi,
\]
completing the proof.
\end{proof}




\end{document}